\documentclass[a4papaer,twosided]{amsart}

\usepackage{latexsym}
\usepackage[english]{babel}
\usepackage{fancyhdr}
\usepackage[mathscr]{eucal}
\usepackage{amsmath}
\usepackage{mathrsfs}
\usepackage{mathtools}
\usepackage{amsthm}
\usepackage{amsfonts}
\usepackage{amssymb}
\usepackage{amscd}
\usepackage{bbm}
\usepackage{graphicx}
\usepackage{subcaption}
\usepackage{graphics}
\usepackage{latexsym}
\usepackage{color}
\usepackage{enumerate}
\usepackage{stmaryrd}

\usepackage{pifont}
\usepackage{booktabs} % for '\midrule' macro

\newcommand{\K}[2]{\mathcal{K}\left(#1,#2\right)}

\newcommand{\Kc}[2]{\overline{\K{#1}{#2}}}

\newcommand{\Kint}[2]{\mathscr{I}_\mathcal{K}\left(#1,#2\right)}

\newcommand{\ran}{\mathrm{ran}}

\newcommand{\Bop}{\mathscr{B}(\cH)}

\newcommand{\cH}{\mathcal{H}}
\newcommand{\dom}[1]{\mathcal{D}\left(#1\right)}

\newcommand{\0}{\mathbb{O}}
\newcommand{\1}{\mathbbm{1}}

\newcommand{\R}{\mathbb{R}}
\newcommand{\C}{\mathbb{C}}
\newcommand{\N}{\mathbb{N}}

\newcommand{\rd}{\mathrm{d}}
\newcommand{\perpV}{\perp_V}

\theoremstyle{plain}
\newtheorem{theorem}{Theorem}[section]
\newtheorem{lemma}[theorem]{Lemma}
\newtheorem{corollary}[theorem]{Corollary}
\newtheorem{proposition}[theorem]{Proposition}

\theoremstyle{definition}
\newtheorem{definition}[theorem]{Definition}

\newtheorem{remark}[theorem]{Remark}

\numberwithin{equation}{section}

\newcommand\blfootnote[1]{
    \begingroup
    \renewcommand\thefootnote{}\footnote{#1}
    \addtocounter{footnote}{-1}
    \endgroup
}

\begin{document}

\title[Structural properties of Krylov subspaces and Krylov solvability]{Structural properties of Krylov subspaces, Krylov solvability, and applications to unbounded self-adjoint operators}
\author[N.~Caruso]{No\`{e} Angelo Caruso}
\address[N.~Caruso]{Dipartimento di Scienze Umane, L'Universit\`{a} degli Studi ``Link Campus University'' \\ Via del Casale di San Pio V, 44\\ 00165 Roma (ITALIA).}
\email{n.caruso@unilink.it}
\thanks{The author wishes to acknowledge the support of the Italian National Institute for Higher Mathematics INdAM, and wishes to thank Marko Erceg (University of Zagreb, Croatia) and Alessandro Michelangeli (Prince Mohammad Bin Fahd University, Saudi Arabia; University of Bonn, Germany; and TQT Trieste Institute for Theoretical Quantum Technologies, Italy) for many stimulating discussions.}


\blfootnote{Accepted to {\it Banach Journal of Mathematical Analysis}.}

\begin{abstract}
This paper presents a study of the inherent structural properties of Krylov subspaces, in particular for the self-adjoint class of operators, and how they relate with the important phenomenon of `Krylov solvability' of linear inverse problems. Owing to the complexity of the problem in the unbounded setting, recently developed perturbative techniques are used that exploit the use of the weak topology on Hilbert space. We also make a strong connection between the approximation properties of the Krylov subspace and the famous Hamburger problem of moments, in particular the determinacy condition.
\end{abstract}

\maketitle

%\tableofcontents
\section{Introduction}
`Krylov solvability' of an inverse linear problem is an operator-theoretic phenomenon with deep implications for applied and theoretical numerical analysis that has garnered recent attention \cite{CM-krylov-perturbation-2021, CMN-2018_Krylov-solvability-bdd,CM-2019_ubddKrylov,CM-Nemi-unbdd-2019,CM-book-Krylov-2022,C-KrylovNormal-2022,CM-krylov-INdAMreview-2024}. An inverse linear problem on a Hilbert space $\cH$ is formulated as
\begin{equation}\label{eq:lininv}
Af = g\,, \quad g \in \ran A\,,
\end{equation}
where $A:\cH \to \cH$ is a closed and densely defined operator, and $f$ is a solution to the above problem. As $g \in \ran A$ we call \eqref{eq:lininv} `solvable'. If, additionally, $A$ is injective, we call \eqref{eq:lininv} `well-defined'; and if furthermore we suppose that $A^{-1}$ is continuous and everywhere defined on $\cH$, the problem \eqref{eq:lininv} is called `well-posed'.

In practical circumstances arising from numerical analysis, one searches for a solution(s) $f$ to \eqref{eq:lininv} within the celebrated Krylov subspace by using a wisely-chosen member of the large family of Krylov algorithms. Therefore, it is critical to know under what circumstances does a solution belong (or fail to belong) to the closure of the Krylov subspace. Such an understanding is important before the actual algorithm is selected and run on the particular problem in order to decide whether the treatment by a particular algorithm has the possibility to yield a correct approximation to a solution of \eqref{eq:lininv}. 

More precisely, the Krylov subspace is constructed using the operator $A$ and known datum vector $g$ in the following manner
\begin{equation}\label{eq:KrylovSubspace}
\K{A}{g} := \mathrm{span}\{A^kg \,|\, k \in \N_0\}\,,
\end{equation}
or in other words, $\K{A}{g}$ is the space of all the polynomials in the operator $A$ applied to the vector $g$. Definition~\eqref{eq:KrylovSubspace} only makes sense when $g$ is an $A$-smooth vector, i.e., $g \in C^\infty(A) := \cap_{n \in \N_0} \mathcal{D}(A^n)$. In purely operator-theoretic jargon, $\K{A}{g}$ is known as the cyclic subspace of the operator $A$ with respect to the vector $g$.

If the inverse linear problem exhibits the existence of a solution in the closure of $\K{A}{g}$, we call such a solution a \emph{Krylov solution}, and we say that \eqref{eq:lininv} is \emph{Krylov solvable}. Until recently \cite{CMN-2018_Krylov-solvability-bdd,CM-2019_ubddKrylov,CM-Nemi-unbdd-2019,CM-book-Krylov-2022}, limited attention has been given to this topic, and mainly in the context of studying the convergence of specific algorithms (e.g., GMRES, CG, LSQR, etc.) to a solution of a inverse linear problem typically with specific assumptions on the operator (e.g., $A$ is bounded, positive, purely discrete spectrum, etc)--see  \cite{Nemirovskiy-Polyak-1985,Nemirovskiy-Polyak-1985-II,Engl-Hanke-Neubauer-1996,Hanke-ConjGrad-1995,Karush-1952,Campbell-Ipsen-1996-I,Campbell-Ipsen-1996-II,CN-2018}.

The recent studies \cite{CMN-2018_Krylov-solvability-bdd,CM-2019_ubddKrylov} have indicated a strong link between the Krylov solvability properties of an inverse linear problem and the inherent structural properties of the Krylov subspace, even though the latter may be difficult to access and therefore analyse. In particular, a construct known as the `Krylov intersection' has been identified as a critical structure to determine whether or not the inverse linear problem is Krylov solvable (see \cite{CMN-2018_Krylov-solvability-bdd,CM-2019_ubddKrylov} for further details). In the bounded operator setting, the analysis of structural properties of Krylov subspaces simplifies owing to the absence of domain issues: indeed any vector $x \in \cH$ is in the domain of the operator, and therefore we do not have to account for possible domain issues when studying closed subspaces under the action of the operator.

Though the structural properties of Krylov subspaces in the unbounded setting and their link to the Krylov solvability of the linear inverse problem has already been discussed in a previous study \cite{CM-2019_ubddKrylov}, here we re-analyse the problem via an analogous route by changing the topology of the closure of the Krylov subspace by using the graph norm. This permits us to avoid complications that arise in establishing links between Krylov solvability and the structural properties of the Krylov subspace as we focus on the class of closed operators acting on a Hilbert space $\cH$. Such complications naturally arise when using the ambient topology of $\cH$ due to domain issues when we consider closures of subsets of the domain of an unbounded operator. Indeed, there are complications in \cite{CM-2019_ubddKrylov} that manifest in the additional assumption that certain solution(s) to \eqref{eq:lininv}, when projected onto $\Kc{A}{g}$, should remain within the domain of the operator. In both practical and analytical situations, this projection criterion may be difficult to verify.

In this article for the first time we also unmask the strong link between the structural properties of Krylov subspaces arising from self-adjoint operators and the famous problem of moments on the real line, or Hamburger moment problem, and orthogonal polynomials. The Hamburger moment problem and orthogonal polynomials is a classical area of analysis that also has intimate links to the theory of symmetric operators and their self-adjoint extensions (see \cite{schmu_moment} for an excellent overview). This link between Krylov subspaces generated by self-adjoint operators and the Hamburger moment problem turns out to be of paramount importance for describing the approximation properties of the Krylov subspace, in particular when it is isometrically isomorphic to an $L^2$ measure space. Indeed, the existence of a unique solution to the Hamburger moment problem indicates that the Krylov subspace possesses good approximation properties, as outlined in Remark~\ref{rem:Kapprox_HamburgerDet}. These approximation properties permit us to investigate further the structural properties of Krylov subspaces, and in particular the notion of Krylov intersection as well as the so-called `Krylov core condition' (see \cite{CM-2019_ubddKrylov} for details).

This paper is split into several parts: Section~\ref{sec:KrylovInt} deals with the reformulation of the Krylov intersection (see Definition~\ref{def:KrylovInt} and Remark~\ref{rem:graphnormclosure_advantage}) and the Krylov solvability of the inverse linear problem under a change in topology using the graph norm of the operator on its domain. We mention in particular Theorem~\ref{th:KrylovInt} that establishes the importance of this reformulated Krylov intersection and establishes the conditions that guarantee the Krylov solvability of \eqref{eq:lininv}. Moreover, Theorem~\ref{th:KrylovInt} avoids the drawbacks of similar propositions present in \cite{CM-2019_ubddKrylov} that require additional assumptions to establish the Krylov solvability. This theorem certainly reaffirms and strengthens the importance of structural properties of Krylov subspaces in determining the Krylov solvability of the inverse linear problem \eqref{eq:lininv}.

Section~\ref{sec:review_struct_perturb} is split into several subsections. Sections~\ref{subsec:review_struct} and \ref{subsec:review_struct_weakgap} review structural notions of Krylov subspaces and the concept of the `weak-gap', respectively. Section~\ref{subsec:review_struct_convergenceweakgap} reviews the convergence properties of linear subspaces under the weak-gap metric between closed subspaces, as well as presenting several new results. Among the new results worth mentioning are Propositions~\ref{prop:weakconvergence_orthog} and \ref{prop:convergenceofprojections}. Proposition~\ref{prop:weakconvergence_orthog} establishes conditions for the convergence in the weak-gap metric of the orthogonal complements of converging subspaces. This result is then exploited further in Proposition~\ref{prop:convergenceofprojections} that establishes the strong operator topology convergence of the associated projection operators of the converging subspaces. Proposition~\ref{prop:weakconvergence_orthog} is also further exploited in Section~\ref{sec:selfadjoint_perturbations} for a key result, Theorem~\ref{th:Kreduced}.

Section~\ref{sec:selfadjoint_struct} is split into two main parts: Section~\ref{subsec:selfadjoint_Hamburger_moment} reviews the basics of the Hamburger moment problem, and Section~\ref{subsec:selfadjoint_struct_Krylov} uses this theory to uncover structural properties of Krylov subspaces arising from self-adjoint operators. We highlight from Section~\ref{subsec:selfadjoint_struct_Krylov}, Propositions~\ref{prop:quasianalytic_determinate}, \ref{prop:Kcore_Hamburgerderteminate}, and \ref{prop:bddvec_invariances}. Proposition~\ref{prop:quasianalytic_determinate}, together with Remarks~\ref{rem:Kapprox_HamburgerDet} and \ref{rem:KisomorphismL2}, establish an isometric isomorphism between the closed Krylov subspace $\Kc{A}{g}$, for self-adjoint $A$, and an $L^2$ measure space under suitable conditions on the vector $g$. This is a natural followup of a previous result \cite[Theorem~7.2]{CM-2019_ubddKrylov}, and is indicative of the `good approximation' properties of $\Kc{A}{g}$ under mild assumptions on the regularity of the vector $g$. Proposition~\ref{prop:Kcore_Hamburgerderteminate} further establishes a very important structural property of the Krylov subspace, known as the Krylov core condition. It is a natural followup and extension of the result \cite[Theorem~7.1]{CM-2019_ubddKrylov}, and solidifies the link between the Hamburger moment problem and the theory presented in this study. Lastly, Proposition~\ref{prop:bddvec_invariances} reveals for the first time new structural information for the Krylov subspace. Under more stringent regularity conditions on the vector $g$, the orthogonal complement of $\K{A}{g}$ in the Hilbert space induced by the \emph{graph norm} scalar product remains contained in the orthogonal complement of $\K{A}{g}$ induced by the scalar product of the \emph{ambient} Hilbert space. This is a crucial piece of structural information that is not at all trivial. It is immediately exploited in Theorem~\ref{th:gbdd_Ksolv_Kinttrivial} to establish Krylov solvability for the inverse linear problem. More importantly, it is used as a critical ingredient to establish Krylov solvability in Theorem~\ref{th:Kreduced} for larger classes of vectors $g$.

Lastly, in Section~\ref{sec:selfadjoint_perturbations} we close the study with a perturbative analysis of Krylov subspaces as well as an analysis of the Krylov solvability arising from self-adjoint inverse linear problems. Worth mentioning in particular are Proposition~\ref{prop:KntoK}, Corollary~\ref{cor:Korthog_limits}, and Theorem~\ref{th:Kreduced}. Proposition~\ref{prop:KntoK} for a self-adjoint operator $A$ establishes the existence of a suitable perturbed sequence of Krylov subspaces $(\Kc{A}{g_n})_{n \in \N}$, such that $g_n \to g$ in the graph norm, and that converge in the weak-gap metric to a well-characterised linear subspace containing $\Kc{A}{g}$. This well-characterised subspace is isometrically isomorphic to an $L^2$ measure space and may actually possess better approximation properties than $\Kc{A}{g}$. This is a rather counter-intuitive result as the proof uses a containment of approximating Krylov subspaces $\Kc{A}{g_n} \subset \Kc{A}{g}$ for each $n \in \N$. Proposition~\ref{prop:KntoK} therefore demonstrates that `simple' approximations of $\Kc{A}{g}$ may actually converge to a subspace possibly containing \emph{more} elements than $\Kc{A}{g}$. Corollary~\ref{cor:Korthog_limits} to Proposition~\ref{prop:KntoK} reveals the convergence properties of the corresponding orthogonal complements of $\Kc{A}{g_n}$. Lastly, our Theorem~\ref{th:Kreduced} establishes for self-adjoint $A$ and suitably regular $g$ the key structural property of the \emph{triviality of the Krylov intersection}, and thus the Krylov solvability for the inverse linear problem \eqref{eq:lininv}. This theorem expands on the study \cite{CM-2019_ubddKrylov} in which the triviality of the Krylov intersection for self-adjoint $A$ was established for a more restrictive class of vectors $g$. It also complements \cite{CM-Nemi-unbdd-2019} where Krylov solvability was established for $A \geqslant \0$ and a suitably general class of vectors $g$, yet without gaining any information on the underlying structure of the Krylov subspace itself.

\subsection{Notation}\label{subsec:notation}~

Throughout this paper, the operator $A:\cH \to \cH$ is a closed operator on a Hilbert space $\cH$ with domain $\dom{A}$. As $A$ is closed, the normed vector space $V := (\dom{A},\Vert \cdot\Vert_A)$ is complete, where 
\begin{equation}\label{eq:graphnorm}
\Vert x\Vert_A^2 := \Vert x\Vert_\cH^2 + \Vert Ax \Vert_\cH^2\,,
\end{equation}
for all $x \in \dom{A}$. By the angled brackets $\langle \cdot, \cdot\rangle$ we denote the standard scalar product on Hilbert space (anti-linear in the first component), and by 
\begin{equation}\label{eq:graphscalar}
\langle x, y\rangle_V := \langle x, y \rangle + \langle Ax, Ay\rangle
\end{equation}
for $x,y \in \dom{A}$, we define the appropriate scalar product on the space $V$ that induces $\Vert\cdot\Vert_A$, thereby making $V$ a Hilbert space.

Throughout the paper, given any $M \subset \cH$, we denote by $\perp$ the standard orthogonal complement in $\cH$,
\[M^\perp := \{x \in \cH \,|\, \langle x, y \rangle=0\,,\, \forall\, y\in M\}\,,\]
and under the further condition that $M \subset V$, by $\perpV$ we denote the space
\[M^{\perpV} := \{x \in V \,|\, \langle x, y\rangle_V = 0 \,,\, \forall\, y \in M\}\,.\]
We also use $\overline{\,\cdot\,}$ to denote closure in the ambient Hilbert space $\cH$ with the norm $\Vert \cdot \Vert_\cH$, and $\overline{\,\cdot\,}^V$ to denote the closure in $V$ with the norm $\Vert \cdot\Vert_A$. We use $\Vert \cdot\Vert_\mathrm{op}$ to denote the standard operator norm from $\cH$ to $\cH$.

It is worth making the remark on the strong and weak closure of $M$ in $\cH$ when $M\subset\cH$ is a linear subspace. Under these conditions the strong and weak closures of $M$ in $\cH$ coincide as $M$ is convex (see \cite[Theorem~3.7]{Brezis-FA-Sob-PDE}). This argument also applies when $M \subset V$ and we consider the strong and weak closures in $V$.

Also true is the fact that for any $M \subset \cH$ (not necessarily linear) the space $M^\perp$ is strongly closed in $\cH$. Furthermore, as $M^\perp$ is a linear subspace and thus convex, it is also weakly closed in $\cH$. Again, a similar argument applies when $M \subset V$ thus showing that also $M^{\perpV}$ is closed in both the strong and weak topology of $V$.

It is well known that the operator $A$, when viewed as acting from the space $V$ (a Hilbert space with \emph{stronger} topology than its embedding space) to $\cH$, is a \emph{bounded} operator, i.e., $A \in \mathcal{B}(V,\cH)$.

\section{The Krylov Intersection}\label{sec:KrylovInt}
In this section we develop the Krylov intersection using the space $V$. We therefore avoid the necessity that the orthogonal projection onto the closed Krylov subspace (in $\cH$) of at least one solution to the inverse linear problem remain in the domain of the operator in order to establish Krylov solvability of the inverse linear problem. To clarify the precise meaning of this point, we state the relevant proposition from \cite{CM-2019_ubddKrylov}.
\begin{proposition}[Proposition~6.5, \cite{CM-2019_ubddKrylov}]\label{prop:Pk_requirements}
For a given Hilbert space $\cH$ let $A$ be a densely defined and closed operator on $\cH$, let $g \in \ran A \cap C^\infty(A)$, and let $f \in \mathcal{D}(A)$ satisfy $Af = g$. Assume furthermore that
\begin{itemize}
	\item[(a)] $A \big( \Kc{A}{g} \cap \mathcal{D}(A) \big) \subset \Kc{A}{g}$,
	\item[(b)] $P_\mathcal{K}f \in \mathcal{D}(A)$, where $P_\mathcal{K}$ is the orthogonal projection operator onto the subspace $\Kc{A}{g}$, and
	\item[(c)] $\Kc{A}{g} \cap A \big( \K{A}{g}^\perp \cap \mathcal{D}(A) \big) = \{0\}$.
\end{itemize}
Then there exists $f_\circ \in \Kc{A}{g} \cap \mathcal{D}(A)$ such that $Af_\circ = g$.
\end{proposition}
Proposition~\ref{prop:Pk_requirements} gives conditions that guarantee the Krylov solvability of the inverse linear problem in the unbounded setting, yet there is the rather strong requirement (b) that the projection onto $\Kc{A}{g}$ of at least one solution to the inverse linear problem remain within the domain of the operator $A$; something that in practice is difficult to verify. The advantage of the theory in this present article is that this assumption (b) in Proposition~\ref{prop:Pk_requirements} is bypassed altogether.

We now proceed to prove some technical propositions and a lemma that will be needed in proving one of the main results of this section, namely Theorem~\ref{th:KrylovInt}.

\begin{proposition}\label{prop:invariance}
Let $A:\cH \to \cH$ be a closed linear operator, and $g \in C^\infty(A)$. Then 
\begin{equation}
A \Kc{A}{g}^V \subset \Kc{A}{g}\,.
\end{equation}
\end{proposition}
\begin{proof}
Indeed, 
\[A \K{A}{g} \subset A \Kc{A}{g}^V \subset \overline{A \K{A}{g}} \subset \Kc{A}{g}\,,\]
the second inclusion being a consequence of the continuity of A as an operator from $V$ to $\cH$ (see \cite[Theorem~18.1]{Munkres-Topology}).
\end{proof}

\begin{lemma}\label{lem:lemma1}
Let $A: \cH \to \cH$ be a closed linear operator, and let $0 \in \rho(A)$. Then, in addition to $A^{-1} \in \Bop$, we have $A^{-1} \in \mathcal{B}(\cH,V)$. The converse also holds, i.e., if $A^{-1}$ exists and $A^{-1} \in \mathcal{B}(\cH,V)$, then $0 \in \rho(A)$.
\end{lemma}
\begin{proof}
If $0 \in \rho(A)$, we have that $A$ is injective on its domain and surjective. As $A: V \to \cH$ is a bounded bijection between $V$ and $\cH$, from the open mapping theorem we immediately have that $A$ is a homeomorphism between $V$ and $\cH$.

We now show the converse statement. Indeed assuming that $A^{-1}$ exists and is bounded from $\cH$ to $V$, we see that 
\[\Vert A^{-1}\Vert_\mathrm{op} = \sup_{\Vert x\Vert_\cH \leqslant 1} \Vert A^{-1}x \Vert_\cH \leqslant \sup_{\Vert x\Vert_\cH \leqslant 1} \Vert A^{-1}x\Vert_A = \Vert A^{-1}\Vert_{\mathcal{B}(\cH,V)} \,.\]
\end{proof}

\begin{proposition}\label{prop:technicalprop}
Let $A$ be a closed, injective operator on $\cH$, and $g \in \ran A \cap C^\infty (A)$. Let $f \in \dom{A}$ be the unique solution to $Af = g$.
\begin{itemize}
	\item[(i)] If $f \in \Kc{A}{g}^V$, then $A \Kc{A}{g}^V$ is dense in $\Kc{A}{g}$.
	\item[(ii)] Assume further that $0 \in \rho(A)$. Then $f \in \Kc{A}{g}^V$ if and only if $A\Kc{A}{g}^V$ is dense in $\Kc{A}{g}$.
\end{itemize}
\end{proposition}
\begin{proof}
We prove (i). Indeed $A \K{A}{g} \subset A \Kc{A}{g}^V \subset \Kc{A}{g}$ and $A^kg \in A \Kc{A}{g}^V$ for all $k \in \N_0$ as $f \in \Kc{A}{g}^V$. Thus $\K{A}{g} \subset A\Kc{A}{g}^V$. Yet, $A\Kc{A}{g}^V \subset \Kc{A}{g}$, so the conclusion follows.

To prove (ii) we only consider the `if' implication. Let $A\Kc{A}{g}^V$ be dense in $\Kc{A}{g}$, and choose a sequence $(v_n)_{n \in \N} \subset \Kc{A}{g}^V$ such that $\Vert Av_n - g\Vert_\cH \xrightarrow{n \to \infty} 0$. As $A^{-1} \in \mathcal{B}(\cH,V)$ by Lemma~\ref{lem:lemma1} we see that $(v_n)_{n \in \N}$ is Cauchy in $V$, and therefore $v_n \xrightarrow{n \to \infty} v \in \Kc{A}{g}^V$. Furthermore,
\[\Vert A^{-1}(Av_n - g)\Vert_{A} = \Vert v_n - f\Vert_{A} \leqslant \Vert A^{-1}\Vert_{\mathcal{B}(\cH,V)} \Vert Av_n - g\Vert_\cH \xrightarrow{n\to\infty} 0\,.\]
Therefore, $v = f$ and so $f \in \Kc{A}{g}^V$.
\end{proof}

We now define a structure central to the notion of Krylov solvability known as the \emph{Krylov intersection}. We do not follow the original definition as provided in \cite{CM-2019_ubddKrylov}, but instead we propose a new structure in such a way as to avoid certain limitations in \cite{CM-2019_ubddKrylov} that naturally arise as a consequence of the underlying operator not necessarily being defined on all of $\cH$.

\begin{definition}\label{def:KrylovInt}
Let $A$ be a closed linear operator on $\cH$, and $g \in C^\infty(A)$. We define the Krylov intersection $\Kint{A}{g}$ as follows
\begin{equation}\label{eq:KrylovInt}
\Kint{A}{g} := \Kc{A}{g} \cap A \big(\K{A}{g}^{\perpV}\big)\,.
\end{equation}
\end{definition}

\begin{theorem}\label{th:KrylovInt}
Let $A$ be a closed linear operator on $\cH$, and let $g \in C^\infty(A)$ be such that $g \in \ran A$. We have the following.
\begin{itemize}
	\item[(i)] If $\Kint{A}{g} = \{0\}$, then there exists a solution $f$ to the inverse linear problem $A f = g$ such that $f \in \Kc{A}{g}^V$.
	\item[(ii)] If in addition we have that $0 \in \rho(A)$, then the solution $f$ to the inverse linear problem $Af=g$,  $f \in \Kc{A}{g}^V$ if and only if $\Kint{A}{g} = \{0\}$.
\end{itemize}
\end{theorem}

\begin{remark}\label{rem:comparison_unbddpaper}
In comparison to Proposition~\ref{prop:Pk_requirements} we have avoided the need to ensure that certain solution(s) to the inverse linear problem remain in the domain under the action of the projection operation onto the closure of $\K{A}{g}$ in $\cH$, a condition which may be theoretically and practically difficult to verify. This is overcome by using the graph norm topology (i.e., $V$) and projection operators on $V$ as will become apparent in the course of the proof of Theorem~\ref{th:KrylovInt}. This presents an advantage to the structural analysis presented in the previous study \cite{CM-2019_ubddKrylov}. 
\end{remark}

\begin{remark}\label{rem:graphnormclosure_advantage}
Though by some measure it may seem a disadvantage to search for solutions of the inverse linear problem in the closure of $\K{A}{g}$ in the space $V$ rather than $\cH$, as $\Kc{A}{g}^V \subset \Kc{A}{g}$, it is worth noting a practical advantage of doing so. In the application of Krylov based algorithms, a sequence $(f_n)_{n \in \N}$ is constructed in $\K{A}{g}$ to approximate a solution $f$ to \eqref{eq:lininv}. In practical situations the solutions are a-priori unknown, and we only have access to the so-called residual $Af_n - g$ for each $n \in \N$. Therefore the knowledge that this vanishes is practically relevant in order to determine conditions under which to stop the algorithm as it has come to some sort of `acceptable' level of convergence. Indeed, the existence of a solution $f$ to \eqref{eq:lininv} in $\Kc{A}{g}^V$ always guarantees the existence of a sequence $(\tilde{f}_n)_{n \in \N}$ in $\K{A}{g}$ that converges in graph norm to $f$, and thus the vanishing of $A\tilde{f}_n - g$. This is, however, not necessarily true from the sole information that there exists a solution $f$ to \eqref{eq:lininv} in $\Kc{A}{g}$.
\end{remark}

\begin{proof}[Proof of Theorem~\ref{th:KrylovInt}]
We begin by proving (i). Let $P_\mathcal{K}^V:V \to V$ denote the orthogonal projection operator (in terms of the scalar product $\langle \cdot, \cdot \rangle_V$) onto the subspace $\Kc{A}{g}^V \subset V$, and let $\1_V$ be the identity on $V$. This operator $P_\mathcal{K}^V$ is continuous from $V$ to $V$. We may always decompose a solution $f$ to $Af = g$ as $f = P_\mathcal{K}^Vf + (\1_V - P_\mathcal{K}^V)f$. Then 
\[Af = g = A P_\mathcal{K}^Vf + A (\1_V - P_\mathcal{K}^V)f\,.\] 
As $g \in \Kc{A}{g}^V \subset \Kc{A}{g}$ and $A P_\mathcal{K}^Vf \in \Kc{A}{g}$ (Proposition~\ref{prop:invariance}), we have $A (\1_V - P_\mathcal{K}^V)f \in \Kc{A}{g}$. Yet, we also know that $(\1_V - P_\mathcal{K}^V)f \in \K{A}{g}^{\perpV}$, so that $A (\1_V - P_\mathcal{K}^V)f \in A \big( \K{A}{g}^{\perpV}\big)$ and thus $A (\1_V - P_\mathcal{K}^V)f = 0$. This implies $A P_\mathcal{K}^V f = g$, and therefore $\tilde{f} = P_\mathcal{K}^V f \in \Kc{A}{g}^V$ is a solution to the inverse linear problem.

We now prove the `only if' part of (ii). Let the solution $f$ to $Af = g$ be in $\Kc{A}{g}^V$, and suppose that $w \in \Kint{A}{g}$. Then there exists a unique $v \in \K{A}{g}^{\perpV}$ such that $w = Av$. By Proposition~\ref{prop:technicalprop} we know that $A \Kc{A}{g}^V$ is dense in $\Kc{A}{g}$, and therefore there exists a sequence $(x_n)_{n\in\N} \subset \Kc{A}{g}^V$ such that $\Vert Ax_n - w\Vert_\cH \to 0$ as $n \to \infty$. Thus
\[\Vert x_n - v\Vert_{A} = \Vert A^{-1}(Ax_n - w)\Vert_{A} \leqslant \Vert A^{-1}\Vert_{\mathcal{B}(\cH,V)}\Vert Ax_n - w\Vert_\cH \xrightarrow{n\to\infty} 0\,.\]
Therefore $x_n \to v$ in $\Vert \cdot\Vert_{A}$, and this implies $v \in \Kc{A}{g}^V$ as well, thus $v = 0$.
\end{proof}

We shall now present a numerical indicator to clearly show us whether subspaces are trivially or non-trivially intersecting. We include such a tool as a first step towards a practical numerical indicator of the triviality of $\Kint{A}{g}$ that can be monitored in computations.
%\begin{definition}\label{def:subspace_dist1}
%Let $M \subset \cH$ and $N \subset \cH$ be subspaces. We define a `separation' between the two subspaces as
%\begin{equation}\label{eq:subspaceDIST}
%\mathrm{sep}(M,N) := \inf_{\substack{u \in M\,, \\ \Hnorm{u} = 1}} \inf_{\substack{v \in N\,,\\ \Hnorm{v}=1}} \Hnorm{u - v}\,.
%\end{equation}
%\end{definition}
%
%\begin{lemma}\label{lem:subspaceDIST1}
%Let $M \subset \cH$ and $N \subset \cH$ be subspaces. If $\mathrm{sep}(M,N) > 0$, then $M \cap N = \{0\}$.
%\end{lemma}
%\begin{proof}
%Suppose by contrapositive that there exists $w \in M \cap N$ with $\Hnorm{w} = 1$. Then,
%\[\inf_{\substack{v \in N\,,\\ \Hnorm{v} = 1}} \Hnorm{w - v} = 0\,,\]
%which implies $\mathrm{sep}(M,N) = 0$.
%\end{proof}
%
%\begin{remark}\label{rem:rem1}
%This measure of a separation between the subspaces is actually too crude for our liking. Ideally we wish to have a distance function that is 0 if and only if $M \cap N = \{0\}$. That will be the idea of the next construction.
%\end{remark}

\begin{definition}\label{def:subspace_dist2}
Let $M$ and $N$ be subspaces of $\cH$. We define the range of `separations' between the subspaces as
\begin{equation}\label{eq:subspaceRANGE}
\mathcal{S}(M,N) := \left\{ \left. \inf_{\substack{u \in M\,,\\ \Vert u\Vert_\cH = 1}} \Vert u - v\Vert_\cH \, \right \vert \, v \in N\,, \, \Vert v\Vert_\cH = 1 \right\}\,.
\end{equation}
\end{definition}

\begin{remark}\label{rem:subspacerange}
It is clear that $\mathcal{S}(M,N) \subset [0,\sqrt{2}]$.
\end{remark}

\begin{proposition}\label{prop:subspaceDIST2}
Let $M$ and $N$ be subspaces of $\cH$. Then $\overline{M} \cap N = \{0\}$ if and only if $0 \notin \mathcal{S}(M,N)$.
\end{proposition}

\begin{proof}
We begin with the `if' statement first. By contrapositive suppose that there is some $w \in \overline{M} \cap N$ with $\Vert w\Vert_\cH = 1$. Fix $\varepsilon > 0$. There exists some $u_\varepsilon \in M$ with $\Vert u_\varepsilon\Vert_\cH = 1$ such that $\Vert u_\varepsilon - w\Vert_\cH < \varepsilon$ as $w \in \overline{M}$. Therefore, as $\varepsilon$ may be chosen arbitrarily small, it is immediate that
\[\inf_{\substack{u \in M\,,\\ \Vert u\Vert_\cH=1}} \Vert u - w\Vert_\cH = 0\,,\]
which implies that $0 \in \mathcal{S}(M,N)$.

We now complete the proof by showing the `only if' part. We also prove this part by contrapositive, so we assume that $0 \in \mathcal{S}(M,N)$ and we shall prove that $\overline{M}\cap N \supsetneq \{0\}$. Indeed, take some $v \in N$ such that $\Vert v\Vert_\cH = 1$ and
\[\inf_{\substack{u \in M\,, \\ \Vert u\Vert_\cH = 1}} \Vert u - v\Vert_\cH = 0\,.\]
For all $n \in \N$, we choose some $u_n \in M$ such that $\Vert u_n\Vert_\cH = 1$ and $\Vert u_n - v\Vert_\cH \leqslant \frac{1}{n}$. Then
\[\Vert u_n - u_m\Vert_\cH \leqslant \Vert u_n - v\Vert_\cH + \Vert u_m - v\Vert_\cH \leqslant \frac{1}{n} + \frac{1}{m}\,,\]
meaning that $(u_n)_{n \in \N} \subset M$ is a Cauchy sequence, and therefore $u_n \xrightarrow{n\to\infty} u \in \overline{M}$. Clearly $u = v$, meaning that $v \in \overline{M} \cap N$.
\end{proof}

This has an immediate application to the Krylov intersection, as the next theorem shows.
\begin{theorem}\label{th:trivialKint_distances}
Let $A: \cH \to \cH$ be a closed linear operator on Hilbert space $\cH$, and $g \in C^\infty(A)$. Then $\Kint{A}{g} = \{0\}$ if and only if \[0 \notin \mathcal{S}\Big(\K{A}{g}, A \big(\K{A}{g}^{\perpV}\big)\Big)\,.\]
\end{theorem}
\begin{proof}
Use Proposition~\ref{prop:subspaceDIST2}, and the fact that $\K{A}{g}$, $A \big(\K{A}{g}^{\perpV}\big)$ are linear subspaces.
\end{proof}

\section{A review of structural and perturbative properties of Krylov subspaces}\label{sec:review_struct_perturb}
In this section we review some of the important structural and perturbative properties of Krylov subspaces that were revealed in the studies \cite{CMN-2018_Krylov-solvability-bdd,CM-2019_ubddKrylov,CM-krylov-perturbation-2021} and the monograph \cite{CM-book-Krylov-2022}. We also present the new Proposition~\ref{prop:convergenceofprojections} linking the convergence of subspaces in the weak-gap metric to the strong operator topology convergence of their orthogonal projection operators. This material is mainly of a preparatory nature for the applications presented in Section~\ref{sec:selfadjoint_perturbations}.

\subsection{Krylov core condition, Krylov reducibility}\label{subsec:review_struct}~

Here we review structural properties more relevant to this study that are related to Krylov subspaces (aside from the Krylov intersection), as they were introduced in the recent study \cite{CM-2019_ubddKrylov} (where proofs are presented fully). 

We begin by defining `Krylov reducibility' in a general unbounded operator setting.
\begin{definition}[Definition~ 5.8, \cite{CM-2019_ubddKrylov}]\label{def:generalised_Kred}
Let $A$ be a densely defined and closed operator on a Hilbert space $\cH$ and let $g \in C^\infty(A)$. $A$ is said to be $\K{A}{g}$-reduced in the generalised sense (or informally, \emph{Krylov reduced}) when 
\begin{equation}\label{eq:defKred}
\begin{split}
A \big(\Kc{A}{g} \cap \mathcal{D}(A)\big) \subset \Kc{A}{g} \,,\\
A \big(\K{A}{g}^\perp \cap \mathcal{D}(A)\big) \subset \K{A}{g}^\perp\,.
\end{split}
\end{equation} 
\end{definition}
This phenomenon may not always be present. Indeed, already at the level of a bounded everywhere defined operator, the second inclusion does not necessarily hold (see \cite{CMN-2018_Krylov-solvability-bdd} for an example), though the first inclusion is always guaranteed for $A$ a bounded, everywhere defined operator. Also, the first inclusion may be false in the unbounded setting and in such cases we have the notion of `Krylov escape'. An explicit example of the Krylov escape phenomenon was constructed in \cite[Example~3]{CM-2019_ubddKrylov}. As opposed to the Krylov escape, the orthogonal complement of the Krylov subspace remains invariant under the action of the adjoint in the sense of the following proposition.
\begin{proposition}[Lemma~5.1, \cite{CM-2019_ubddKrylov}]\label{prop:invarianceunderadjoint}
Let $A$ be a densely defined and closed operator acting on a Hilbert space $\cH$ and let $g \in C^\infty(A)$. Then
\begin{equation}\label{eq:invarianceunderadjoint}
A^* \big(\K{A}{g}^\perp \cap \mathcal{D}(A^*)\big) \subset \K{A}{g}^\perp\,.
\end{equation}
\end{proposition}

We now turn to another important structural property called the `Krylov core condition'.
\begin{definition}[Definition~5.4, \cite{CM-2019_ubddKrylov}]\label{def:Krylovcore}
Let $A$ be a densely defined closed operator acting on a Hilbert space $\cH$ and let $g \in C^\infty(A)$. The operator vector pair $(A,g)$ is said to satisfy the `Krylov core condition' if $\Kc{A}{g}^V = \Kc{A}{g} \cap \mathcal{D}(A)$.
\end{definition}

The Krylov core condition turns out to be a sufficient condition to avoid the occurrence of the Krylov escape phenomenon, as seen from Proposition~\ref{prop:invariance} and also from \cite[Lemma~5.6]{CM-2019_ubddKrylov}. It is also an important property that is exploited further in Section~\ref{sec:selfadjoint_perturbations} to establish important structural properties of Krylov subspaces in Theorem~\ref{th:Kreduced}.

\subsection{Weak-gap metric}\label{subsec:review_struct_weakgap}~

In this subsection we review the `weak-gap metric' that has been used to study the effects of perturbations in operator and or datum vector on the Krylov subspace \cite{CM-krylov-perturbation-2021}.

Naturally, the unit ball $B_\cH$ of \emph{any} separable Hilbert space $\cH$ is metrisable in the weak topology (see \cite[Theorem~3.29]{Brezis-FA-Sob-PDE} for further details) with a metric $\rho_w(x,\,y) := \Vert x - y\Vert_{w}$ such that $\Vert x\Vert_{w} \leqslant \Vert x\Vert_\cH$. More precisely, for a dense countable collection of vectors $(v_n)_{n \in \N} \subset B_\cH$, we have
\begin{equation}\label{eq:weaknorm}
\Vert x\Vert_{w} := \sum_{n = 1}^\infty \frac{1}{2^n} \vert \langle v_n, x\rangle_\cH \vert\,,
\end{equation}
where $\langle \cdot, \cdot\rangle_\cH$ is the scalar product on $\cH$ corresponding to the norm $\Vert\cdot\Vert_\cH$. Therefore the weak convergence of a sequence of vectors $(u_n)_{n \in \N} \subset B_\cH$ is tantamount to their convergence in the metric $\rho_w$.

The weak-gap metric applies not to individual vectors, but to the class of non\-empty weakly closed \emph{sets} contained in $B_\cH$
\begin{equation}\label{eq:BHclosedweak}
\mathscr{C}_w(\cH) := \{C \subset B_\cH \,|\, C \neq \emptyset\,, C \text{ weakly closed}\}\,.
\end{equation}
For $C,D \in \mathscr{C}_w(\cH)$ we have
\begin{equation}\label{eq:dw}
d_w(C,D) := \sup_{u \in C} \inf_{v \in D} \Vert u - v\Vert_{w}\,,
\end{equation}
which then gives us the \emph{weak-gap metric} on $\mathscr{C}_w(\cH)$
\begin{equation}\label{eq:dwhat}
\hat{d}_w(C,D) := \max \{ d_w(C,D), d_w(D,C)\}\,.
\end{equation}

We now list some important properties of the weak-gap metric on $\mathscr{C}_w(\cH)$ (without proof).
\begin{theorem}[Theorem~5.1, \cite{CM-krylov-perturbation-2021}]\label{th:fundamentalproperties_dwhat}
Let $\cH$ be a separable Hilbert space.
\begin{itemize}
	\item[(i)] $\hat{d}_w$ is a metric on $\mathscr{C}_w(\cH)$.
	\item[(ii)] The metric space $(\mathscr{C}_w(\cH), \hat{d}_w)$ is complete and compact.
	\item[(iii)] If $\hat{d}_w(C_n,C) \xrightarrow{n \to \infty} 0$ for some element $C$ and a sequence $(C_n)_{n \in \N}$ in $\mathscr{C}_w(\cH)$, then
		\begin{equation}\label{eq:weakgaplimit}
		C = \{u \in B_\cH \,|\, u_n \rightharpoonup u \text{ for a sequence } (u_n)_{n \in \N} \text{ with } u_n \in C_n \, \forall \, n \in \N\}.
		\end{equation}
	\item[(iv)] Suppose $f:\cH \to \cH$ is a weakly closed and weakly continuous map such that $f(B_\cH) \subset B_\cH$. If $\hat{d}_w(C_n,C) \xrightarrow{n \to \infty} 0$ for an element $C$ and sequence $(C_n)_{n \in \N}$ in $\mathscr{C}_w(\cH)$, then $\hat{d}_w(f(C_n),f(C)) \xrightarrow{n \to \infty} 0$
\end{itemize}
\end{theorem}

Before moving on to the application of the weak-gap metric to linear subspaces, we shall list some more properties that are useful later.
\begin{proposition}[Lemma~5.3, \cite{CM-krylov-perturbation-2021}]\label{prop:dwhat_triangleineq}
Let $C,D,E \in \mathscr{C}_w(\cH)$ for some separable Hilbert space $\cH$. Then
\begin{gather}
d_w(C,D) = 0 \text{ if and only if } C \subset D\,,\\
\hat{d}_w(C,D) = 0 \text{ if and only if } C = D\,,\\
d_w(C,E) \leqslant d_w(C,D) + d_w(D,E)\,,\\
\hat{d}_w(C,E) \leqslant \hat{d}_w(C,D) + \hat{d}_w(D,E)\,.
\end{gather}
\end{proposition}

\subsection{Convergence of linear subspaces in the weak-gap metric}\label{subsec:review_struct_convergenceweakgap}~

In order to exploit the weak-gap metric for closed linear subspaces of $\cH$, we describe convergence on the subsets naturally induced by the corresponding unit balls of the \emph{subspaces} of $\cH$ as elements of $(\mathscr{C}_w(\cH),\hat{d}_w)$. More precisely, given closed linear subspaces $M,N$ of a separable Hilbert space $\cH$, we identify by definition
\begin{equation}\label{eq:dwhat_subspace}
\hat{d}_w(M,N) \equiv \hat{d}_w(B_M,B_N)\,,
\end{equation}
with the right side defined by \eqref{eq:dwhat} and $B_M = M \cap B_\cH$, $B_N = N \cap B_\cH$. We write $\hat{d}_w(M_n,M) \xrightarrow{n \to \infty} 0$ to mean that $\hat{d}_w(B_{M_n},B_M) \xrightarrow{n \to \infty} 0$ for some $M$ and sequence $(M_n)_{n \in \N}$ closed linear subspaces of $\cH$. This naturally provides a metric topology and notion of convergence on the collection of closed subspaces
\begin{equation}\label{eq:linearclosed}
\mathscr{S}(\cH) := \{M \subset \cH \,|\, M \text{ closed linear subspace}\}\,,
\end{equation}
thus giving us the metric space $(\mathscr{S}(\cH),\hat{d}_w)$.

We now list the most important properties of $(\mathscr{S}(\cH),\hat{d}_w)$.
\begin{proposition}[Lemmas~6.1, 6.2, Proposition~6.3, \cite{CM-krylov-perturbation-2021}]\label{prop:SHdw_properties}
Let $\cH$ be a separable Hilbert space.
\begin{itemize}\label{prop:properties_dwhat_linearsubspace}
	\item[(i)] $(\mathscr{S}(\cH),\hat{d}_w)$ is a metric space.
	\item[(ii)] $(\mathscr{S}(\cH),\hat{d}_w)$ is not complete.
	\item[(iii)] If $\hat{d}_w(M_n,M) \xrightarrow{n \to \infty} 0$ for some $(M_n)_{n \in \N}$ and $M$ in $\mathscr{S}(\cH)$, then
	\begin{equation}
		M = \{u \in \cH \,|\, u_n \rightharpoonup u \text{ for a sequence } (u_n)_{n \in \N} \text{ with } u_n \in M_n \text{ for all } n \in \N\}\,.
	\end{equation}
\end{itemize}
\end{proposition}

\begin{remark}\label{rem:dwhat_lackofcompleteness}
Though there is a lack of completeness of $(\mathscr{S}(\cH),\hat{d}_w)$, it turns out that $\hat{d}_w$ is still useful as will be revealed in the rest of this article. Moreover, though there are Cauchy sequences $(M_n)_{n \in \N}$ in $(\mathscr{S}(\cH),\hat{d}_w)$ that do not converge to an element of $\mathscr{S}(\cH)$, the corresponding unit balls $B_{M_n}$ in fact \emph{do} still converge to some element of $\mathscr{C}_w(\cH)$ in $\hat{d}_w$ due to the completeness of $(\mathscr{C}_w(\cH),\hat{d}_w)$. In such cases it is clear that the limit in $\mathscr{C}_w(\cH)$ cannot be written as the unit ball of a closed linear subspace of $\cH$.
\end{remark}

Following this review, we now present some new results of a technical nature related to subspaces and convergence in the weak-gap metric.

\begin{lemma}\label{lem:nestedincreasingsets_convergence}
Let $\cH$ be a separable Hilbert space, and let $(C_n)_{n \in \N}$ be a sequence of weakly closed nested sets in $B_\cH$ such that $C_n \subset C_{n + 1}$. Then there exists a weakly closed set $C \subset B_\cH$ such that $C_n \xrightarrow{\hat{d}_w} C$ as $n \to \infty$.
\end{lemma}
\begin{proof}
We use the facts that the collection of weakly closed sets in $B_\cH$ equipped with the metric $\hat{d}_w$ is compact (Theorem~\ref{th:fundamentalproperties_dwhat}(ii)), and both $d_w$ and $\hat{d}_w$ satisfy the triangle inequality (Proposition~\ref{prop:dwhat_triangleineq}). We also note that for any two weakly closed sets $D,E \subset B_\cH$, we have $d_w(D,E) = 0$ if and only if $D \subset E$ (Proposition~\ref{prop:dwhat_triangleineq}).

Let $(C_n)_{n \in \N}$ be a sequence of weakly closed sets as stated in the lemma. From the fact that $(\mathscr{C}_w(\cH), \hat{d}_w)$ is compact, we know there exists a subsequence $(C_{n_k})_{k \in \N}$ that converges to some weakly closed set $C \subset B_\cH$ in the metric $\hat{d}_w$. Let $\varepsilon > 0$, so there exists $N_\varepsilon \in \N$ such that for all $k,k' \geqslant N_\varepsilon$
\begin{align*}
\hat{d}_w(C_{n_k},C_{n_{k'}}) & < \varepsilon \,, \text{ and } \\
\hat{d}_w(C_{n_k},C) & < \varepsilon\,.
\end{align*}

Let $N_\varepsilon' = n_{N_\varepsilon}$ and choose any $n \geqslant N_\varepsilon'$. Then there exist $k,k' \geqslant N_\varepsilon$ such that $C_{n_k} \subset C_n \subset C_{n_{k'}}$. Using the triangle inequality for $d_w$,
\begin{align*}
d_w(C_n,C) & \leqslant d_w(C_n,C_{n_{k'}}) + d_w(C_{n_{k'}}, C) \\
 & = 0 + d_w(C_{n_{k'}},C) < \varepsilon\,, \\
d_w(C,C_n) & \leqslant d_w(C,C_{n_k}) + d_w(C_{n_k},C_n) \\
 & = d_w(C,C_{n_k}) + 0 < \varepsilon\,.
\end{align*}
This implies that $\hat{d}_w(C,C_n) < \varepsilon$ for all $n \geqslant N_\varepsilon'$, and therefore we have the desired convergence.
\end{proof}

\begin{lemma}\label{lem:nesteddecreasingsets_convergence}
Let $\cH$ be a separable Hilbert space and let $(C_n)_{n \in \N}$ be a sequence of weakly closed nested sets in $B_\cH$ such that $C_n \supset C_{n + 1}$. Then there exists a weakly closed set $C \subset B_\cH$ such that $C_n \xrightarrow{\hat{d}_w} C$ as $n \to \infty$.
\end{lemma}
\begin{proof}
The proof is identical to Lemma~\ref{lem:nestedincreasingsets_convergence}
\end{proof}

%As a result of the nesting structure $\Kc{A}{g_n} \subset \Kc{A}{g_{n+1}}$, we have the nesting structure for the orthogonal complement $\K{A}{g_n}^\perp \supset \K{A}{g_{n+1}}^\perp$. As a consequence of Lemma~\ref{lem:nesteddecreasingsets_convergence} we know that the limit, in the sense of $\hat{d}_w$, $\lim_{n \to \infty} \mathcal{K}_n^\perp$ exists as a weakly closed set in $B_\cH$, where we use the shorthand $\mathcal{K}_n^\perp := \K{A}{g_n}^\perp \cap B_\cH$, $\mathcal{K}^\perp := \K{A}{g}^\perp \cap B_\cH$.

\begin{proposition}\label{prop:weakconvergence_orthog}
Let $\cH$ be a separable Hilbert space. Consider $(M_n)_{n \in \N}$ a nested increasing sequence of closed linear subspaces such that
\begin{equation}\label{eq:prop_convorthog}
M_n \xrightarrow{\hat{d}_w} M\,,
\end{equation}
for some $M$ a closed linear subspace of $\cH$. Then $M_n \subset M$ for all $n \in \N$, and $(M_n^\perp)_{n \in \N}$ is a nested decreasing sequence such that
\begin{equation}\label{eq:prop_convorthog2}
M_n^\perp \xrightarrow{\hat{d}_w} M^\perp\,,
\end{equation}
and $M^\perp \subset M_n^\perp$ for all $n \in \N$.
\end{proposition}

\begin{proof}
Consider $(M_n)_{n \in \N}$ and $M$ satisfying the hypotheses of the proposition. We know from the convergence of the $M_n$'s to $M$ (Theorem~\ref{th:fundamentalproperties_dwhat}(iii)) that
\[B_M = \{u \in B_\cH \,|\, u_n \rightharpoonup u \text{ for a sequence } (u_n)_{n \in \N} \text{ with } u_n \in B_{M_n} \, \forall \, n \in \N\}\,.\]
The nested increasing nature of the $M_n$'s means that $M_n \subset M_{n+1}$ for all $n \in \N$. Given any $n_0 \in \N$ and any $u \in B_{M_{n_0}}$, for all $n \geqslant n_0$ let $u_n = u$, and for all $n < n_0$ let $u_n = 0$. Then $(u_n)_{n \in \N}$ is a sequence such that $u_n \in B_{M_n}$ for all $n \in \N$ and $u_n \rightharpoonup u \in B_\cH$ as $n \to \infty$. Therefore, $u \in B_M$, so $B_{M_{n_0}} \subset B_M$. As $n_0$ was arbitrarily chosen, this implies $B_{M_n} \subset B_M$ for all $n \in \N$, so that indeed we have the inclusion
\[M_n \subset M\,, \quad \forall \, n \in \N\,.\]

From the nested increasing nature of the $M_n$'s it follows that
\[M_n^\perp \supset M_{n+1}^\perp\,, \quad \forall \, n \in \N\,,\]
and thus $B_{M_n^\perp} \supset B_{M_{n+1}^\perp}$ for all $n \in \N$. By Lemma~\ref{lem:nesteddecreasingsets_convergence} there exists a weakly closed set $C \subset B_\cH$ such that $\hat{d}_w(B_{M_n^\perp},C) \to 0$ as $n \to \infty$.

We claim that $C = B_{M^\perp}$. To this end, let $v \in C$ so that there exists a sequence $(v_n)_{n \in \N}$ such that $v_n \rightharpoonup v$ and $v_n \in B_{M_n^\perp}$ for all $n \in \N$ (Theorem~\ref{th:fundamentalproperties_dwhat}(iii)). Let $w \in B_M$ so that there exists a sequence $(w_n)_{n \in \N}$ such that $w_n \rightharpoonup w$ and $w_n \in B_{M_n}$ for all $n \in \N$. 

We shall construct a subsequence of weakly closed sets $(B_{M_{n_k}})_{k \in \N}$ along with a sequence $(w'_{n_k})_{k \in \N}$ that has the property $w'_{n_k} \in B_{M_{n_k}}$ for all $k \in \N$ and $w'_{n_k} \to w$ strongly. Indeed, as $w_n \rightharpoonup w$, by Mazur's Theorem \cite[Corollary~3.8]{Brezis-FA-Sob-PDE} there exists a sequence of vectors $(z_k)_{k \in \N}$ that are made up of convex combinations of the $w_n$'s, such that $z_k \to w$ strongly. For each $k \in \N$, $\Vert z_k\Vert_\cH \leqslant 1$ as each $z_k$ is a convex combination of the $w_n$'s and thus remains within the unit ball. As each $z_k$ is a convex combination of the $w_n$'s, it has the form
\[z_k = \sum_{n = 1}^{N_k} t_n^{(k)} w_n\,,\]
where $t_n^{(k)} \geqslant 0$, $N_k \in \N$, and $\sum_{n = 1}^{N_k} t_n^{(k)} = 1$ for all $k \in \N$. Therefore, such a combination exists in the set $B_{M_{N_k}}$ given the nested structure of the Krylov subspaces and the convexity of $B_{M_n}$ for all $n \in \N$. The indexing of the subsequence $w'_{n_k}$ is set as follows: let $k=1$ and choose some number $n_1 \in \N$ such that $z_1 \in B_{M_{n_1}}$. We then set $w'_{n_1} = z_1 \in B_{M_{n_1}}$. We do the same process with $k = 2$, choosing $n_2 > n_1$ which is always possible due to the nested structure of the sequence $(B_{M_n})_{n \in \N}$. Continuing this process, we obtain a subsequence of balls $(B_{M_{n_k}})_{k \in \N}$ with a sequence $(w'_{n_k})_{k \in \N}$ such that $w'_{n_k} \in B_{M_{n_k}}$ for each $k \in \N$, and such that $w'_{n_k} \xrightarrow{k \to \infty} w$ strongly.

Given that $v_{n_k} \rightharpoonup v$, and $v_{n_k} \in B_{M_{n_k}^\perp}$ for all $k \in \N$, we see that
\[0 = \langle v_{n_k}, w'_{n_k}\rangle \xrightarrow{k\to\infty} \langle v, w\rangle\,,\]
implying that $v \perp w$. As the nature of $w$ is arbitrary in $B_M$, this proves that $v \in M^\perp$, and therefore $C \subset B_{M^\perp}$.

It is clear from the fact that $M_n \subset M$ that we have the inclusion $M_n^\perp \supset M^\perp$ for all $n \in \N$. This implies that $C \supset B_{M^\perp}$, and therefore $C = B_{M^\perp}$. We have the convergence
\[M_n^\perp \xrightarrow[n \to \infty]{\hat{d}_w} M^\perp\,,\]
thus completing the proof.
\end{proof}

\begin{remark}\label{rem:nesting_structure}
The nested structure $M_n \subset M_{n+1}$ turned out to be critical for proving the above result, in particular for the use of Mazur's Theorem.
\end{remark}

The next proposition relates convergence of subspaces in the weak-gap metric to the convergence of the orthogonal projections in the strong operator topology.

\begin{proposition}\label{prop:convergenceofprojections}
Consider $\cH$ a separable Hilbert space and let $M$ be a closed linear subspace with $(M_n)_{n \in \N}$ a sequence of closed linear subspaces that satisfy the following conditions:
	\begin{itemize}
		\item[(i)] $M_n \subset M_{n+1}$ for all $n \in \N$,
		\item[(ii)] $M_n \xrightarrow{\hat{d}_w} M$.
		%\item[(iii)] $M_n^\perp \xrightarrow{\hat{d}_w} M^\perp$ as $n \to \infty$.
	\end{itemize}
Then $P_{M_n} \to P_{M}$ in the strong operator topology, where $P_{M_n}$ is the orthogonal projection onto $M_n$, and $P_{M}$ is the orthogonal projection onto $M$.
\end{proposition}

\begin{proof}
We note that under the hypotheses of the proposition $M_n^\perp \xrightarrow{\hat{d}_w} M^\perp$ as $n \to \infty$, and also $M_n^\perp \supset M_{n+1}^\perp \supset M^\perp$, $M_n \subset M$ for all $n \in \N$ by Proposition~\ref{prop:weakconvergence_orthog}. 

First we shall establish the claim that for any $x \in \cH$, $P_{M_n}P_{M_m}x = P_{M_m}P_{M_n}x$ and $P_{M_n}P_{M}x = P_{M}P_{M_n}x$ for all $n,m \in \N$. Indeed, without loss of generality, we assume that $m > n$, so that $P_{M_n}x \in M_m$ due to condition (i) and therefore $P_{M_m}P_{M_n}x = P_{M_n}x$. Also, $x = P_{M_m}x + (\1 - P_{M_m})x$ and $(\1-P_{M_m})x \in M_m^\perp \subset M_n^\perp$. Therefore, $P_{M_n}x = P_{M_n}P_{M_m}x + P_{M_n}(\1 - P_{M_m})x = P_{M_n}P_{M_m}x$, from which $P_{M_m}P_{M_n}x = P_{M_n}x = P_{M_n}P_{M_m}x$. Similarly, $P_{M}P_{M_n}x = P_{M_n}x$ as $P_{M_n}x \in M$; and as $(\1 - P_{M})x \in M^\perp \subset M_n^\perp$ we have $P_{M_n}x = P_{M_n}P_{M}x$. Therefore, $P_{M_n}P_{M}x = P_{M}P_{M_n}x$ as claimed.

Now we let $x \in B_\cH$. We may decompose $x$ as
\[x = P_{M_n}x + P_{M}(\1 - P_{M_n})x + (\1 - P_{M})x\,, \tag{\dag}\]
where $P_{M_n}x$ is the component in $M_n$, $P_{M}(\1 - P_{M_n})x$ is the component in $M$ that is orthogonal to $M_n$, and $(\1 - P_{M})x$ is the component orthogonal to $M$. Each component is also contained in $B_\cH$.

From the compactness of $B_\cH$ in the weak topology \cite[Theorem~3.17]{Brezis-FA-Sob-PDE}, there exists a subsequence of $(P_{M}(\1 - P_{M_n})x)_{n \in \N}$ that weakly converges to some $v \in B_\cH$, i.e.,
\[P_{M}(\1 - P_{M_{n_k}})x \rightharpoonup v\,.\]
%By condition (ii) and the commutativity of $P_{M}$ with $P_{M_{n_k}}$, we see that $v \in M$ from $M_{n_k} \xrightarrow{\hat{d}_w} M$ coupled with Proposition~\ref{prop:SHdw_properties}(iii)
It is clear that $v \in M$ as the sequence $(P_M(\1-P_{M_{n_k}})x)_{k \in \N} \subset M$ and $M$ is weakly closed. Yet, owing to the commutativity of $P_M$ and $P_{M_{n_k}}$, we also know that $P_{M}(\1 - P_{M_{n_k}})x \in M_{n_k}^\perp$ for all $k \in \N$. Therefore, $v \in M^\perp$ from the convergence $M_{n_k}^\perp \xrightarrow{\hat{d}_w} M^\perp$ and Proposition~\ref{prop:SHdw_properties}(iii), as well as the fact that $M^\perp$ is weakly closed. Thus $v = 0$. From the weak convergence $P_M(\1 - P_{M_{n_k}})x \rightharpoonup 0$ we deduce that
\begin{align*}
\Vert P_{M}(\1 - P_{M_{n_k}})x \Vert_\cH^2 & = \bigl \langle P_{M}(\1 - P_{M_{n_k}})x, P_{M}(\1 - P_{M_{n_k}})x \bigr \rangle \\
 & = \bigl \langle P_{M}(\1 - P_{M_{n_k}})x, (\1 - P_{M_{n_k}})x \bigr \rangle\\
 & = \bigl \langle (\1 - P_{M_{n_k}})P_{M}(\1 - P_{M_{n_k}})x, x \bigr \rangle\\
 & = \bigl \langle P_{M}(\1 - P_{M_{n_k}})x, x \bigr \rangle \xrightarrow{k \to \infty} 0\,,
\end{align*}
owing to the commutativity and the fact that $P_{M}$ and $P_{M_{n_k}}$ are orthogonal projections. Therefore,
\[P_{M}(\1 - P_{M_{n_k}})x \xrightarrow[k \to \infty]{\Vert \cdot\Vert_\cH} 0\,,\]
for any $x \in B_\cH$. Using the decomposition $(\dag)$
\[P_{M_{n_k}}x \xrightarrow[k\to\infty]{\Vert\cdot\Vert_\cH} P_M x\,.\]
This result then trivially extends to \emph{any} $x \in \cH$.

We now show that the full sequence $(P_{M_n})_{n \in \N}$ converges to $P_M$ in the strong operator topology. Indeed, let $x \in \cH$ and let $\varepsilon > 0$. There exists some $N_{\varepsilon,x} \in \N$ such that for all $k,k' \geqslant N_{\varepsilon,x}$,
\[\Vert P_{M_{n_k}}x - P_{M_{n_{k'}}}x \Vert_\cH < \varepsilon\,.\]
Choose $n,m \in \N$ such that $n \geqslant n_{N_{\varepsilon,x} + 1}$ and without loss of generality, let $m > n$. We choose some $n_k, n_{k'}$ such that $k,k' \geqslant N_{\varepsilon,x}$ and $n_k < n < m < n_{k'}$. Then,
\[P_{M_m}x - P_{M_n}x = (P_{M_{n_{k'}}}x - P_{M_{n_k}}x) - (P_{M_{n_{k'}}}x - P_{M_m}x) + (P_{M_{n_{k}}}x - P_{M_{n}}x)\,,\]
and
\[P_{M_{n_{k'}}}x - P_{M_{m}}x = P_{M_{n_{k'}}}x - P_{M_{n_{k}}}x + P_{M_{n_{k}}}x - P_{M_{m}}x\,,\]
and from the commutativity of the projections and inclusions $M_m \subset M_{n_{k'}}$, $M_m \supset M_{n_k}$, we have
\begin{align*}
P_{M_{n_{k}}}x - P_{M_{m}}x & = P_{M_{n_{k}}}x - P_{M_{n_{k'}}}P_{M_{m}}x \\
 & = P_{M_{n_{k}}}P_{M_m}x - P_{M_{n_{k'}}}P_{M_{m}}x \\ 
 & = (P_{M_{n_{k}}} - P_{M_{n_{k'}}})P_{M_{m}}x = P_{M_m}(P_{M_{n_{k}}} - P_{M_{n_{k'}}})x\,.
\end{align*}
This implies
\[\Vert P_{M_{n_{k}}}x - P_{M_{m}}x\Vert_\cH = \Vert P_{M_m}(P_{M_{n_{k}}} - P_{M_{n_{k'}}})x\Vert_\cH \leqslant \Vert (P_{M_{n_{k}}} - P_{M_{n_{k'}}})x\Vert_\cH < \varepsilon\,,\]
and therefore
\[\Vert P_{M_{n_{k'}}}x - P_{M_{m}}x\Vert_\cH < 2\varepsilon\,.\]

Using the fact that $P_{M_{n_k}}x = P_{M_{n_k}} P_{M_n}x$ and $P_{M_{n_{k'}}}P_{M_{n}}x = P_{M_{n}}x$ owing to the inclusions $M_n \supset M_{n_k}$ and $M_n \subset M_{n_{k'}}$ respectively, we see that
\begin{align*}
P_{M_{n_{k}}}x - P_{M_{n}}x & = P_{M_{n_{k}}}x - P_{M_{n_{k'}}}P_{M_n}x \\
 & = P_{M_{n_{k}}}P_{M_n}x - P_{M_{n_{k'}}}P_{M_n}x = P_{M_n}(P_{M_{n_{k}}} - P_{M_{n_{k'}}}) x\,.
\end{align*} 
This implies
\[\Vert P_{M_{n_{k}}}x - P_{M_{n}}x\Vert_\cH \leqslant \Vert P_{M_{n_{k}}}x - P_{M_{n_{k'}}}x\Vert_\cH < \varepsilon\,.\]

Putting all these inequalities together
\[\Vert P_{M_n}x - P_{M_m}x \Vert_\cH < 4 \varepsilon\]
which shows us that $(P_{M_{n}}x)_{n \in \N}$ is a Cauchy sequence. Therefore this sequence has a limit in $\cH$ that must be the same as the limit of the convergent subsequence, i.e., $P_M x$. Therefore, we have that for all $x \in \cH$
\[P_{M_n}x \xrightarrow[n \to \infty]{\Vert\cdot\Vert_\cH} P_Mx\,.\]
\end{proof}

\begin{remark}\label{rem:Katogap_Projectionconvergence}
We note that when subspaces converge in the sense of the Kato gap metric (see \cite[Chapter~IV, Section~2]{Kato-perturbation}) in a Hilbert space, this corresponds to the \emph{operator norm} convergence of the projections. In comparison, in Proposition~\ref{prop:convergenceofprojections} we manage to achieve the strong operator topology convergence by having convergence in the weak-gap metric.
\end{remark}

\section{Structural Properties of Self-Adjoint Operators Revisited}\label{sec:selfadjoint_struct}
We already know from previous analysis \cite{CM-2019_ubddKrylov, CM-Nemi-unbdd-2019} that self-adjoint inverse linear problems, under suitable growth conditions on the datum vector $g$, exhibit Krylov-solvability. We aim to show this again in light of the Krylov intersection, thereby unmasking the interplay between structural properties of $\K{A}{g}$ and Krylov solvability that is absent in the analysis \cite{CM-Nemi-unbdd-2019}, and only partially developed in \cite{CM-2019_ubddKrylov}.

It turns out that the structural properties of Krylov subspaces and spectral properties of the operator $A$ in the self-adjoint setting are very strongly interlinked with the moment problem in one dimension. The moment problem, barring more abstract generalisations, seeks to establish a representation (where one exists) of a linear functional acting on polynomials over $\R$ with integration using a positive Radon measure on $\R$. 

This investigation of the link between structural properties of Krylov subspaces and the moment problem in one dimension is a natural follow-up of the study \cite{CM-Nemi-unbdd-2019}. In \cite{CM-Nemi-unbdd-2019} connections are made between the convergence of the approximants to a solution to \eqref{eq:lininv} from the conjugate gradients algorithm and certain properties of determinacy of the moment problem along with properties of orthogonal polynomials on the real line. A very thorough and modern overview of moment problems and orthogonal polynomials is provided in the monograph \cite{schmu_moment}.

An important structural property of Krylov subspaces generated by a self-adjoint operator is the existence of an isomorphism between $\Kc{A}{g}$ and the measure space $L^2\big(\R,\mu_g^{(A)}\big)$, where $\mu_g^{(A)}(\Omega) := \langle E^{(A)}(\Omega)g, g\rangle$ and $E^{(A)}$ is the unique projection valued spectral measure for the operator $A$ and $\Omega \subset \R$ a Borel set (see Remark~\ref{rem:KisomorphismL2}). The presence of this isomorphism is used in Proposition~\ref{prop:Kcore_Hamburgerderteminate} to prove the Krylov core condition.

%There are several classes of important vectors that we list below and shall consider in the following analysis. by considering the class of \emph{bounded} vectors for our datum $g \in \ran A$. Recall that a vector $g \in \mathcal{D}(A)$ is said to be bounded if there exists some $B_g > 0$ such that 
%\begin{equation}\label{eq:def_boundedvec}
%\Hnorm{A^n g} \leqslant B_g^n\,.
%\end{equation}

\subsection{The Hamburger Moment Problem}\label{subsec:selfadjoint_Hamburger_moment}~

We first discuss important results that stem from the Hamburger moment problem in one dimension which are used heavily in our analysis. The Hamburger moment problem is by now a classical area of mathematics and has been very well studied \cite{schmu_moment,Chihara-book-1978,Koornwinder_OrthPoly,Szego-OrthPolyBook,Shohat-Tamarkin_ProblemOfMoments1943}. We begin by discussing some generalities of the moment problem in an abstract setting before focussing our attention specifically on the Hamburger moment problem.

For the general one dimensional moment problem on $\R$ we consider a linear functional $\mathscr{L}$ on a vector space of continuous functions $\mathcal{F}$ over $\R$ and ask: under what conditions does there exist a \emph{positive} Radon measure $\mu$ on $\R$ such that $\mathscr{L}(f) = \int_\R f(x) \, \rd \mu(x)$ for all $f \in \mathcal{F}$? In the case where there does exist such a Radon measure, we say that $\mathscr{L}$ is a \emph{moment functional}.

Typically the moment problem is formulated for the case where $\mathcal{F}$ is the space of all real polynomials on $\R$, and a real-valued linear functional $\mathscr{L}$ either: (i) is defined by a given sequence of real numbers $s \equiv (s_n)_{n \in \N_0}$ by $\mathscr{L}(x^n) := s_n$ for all $n \in \N_0$, or (ii) it defines a sequence of real numbers $s \equiv (s_n)_{n \in \N_0}$ by $s_n := \mathscr{L}(x^n)$ for all $n \in \N_0$. We note immediately the subtle restriction that in this formulation $\mathscr{L}$ is real-valued on the space of real polynomials on $\R$. This removes a certain redundancy in the moment problem: indeed should a linear functional $\widetilde{\mathscr{L}}$ be such that $\widetilde{\mathscr{L}}(p)$ has non-trivial imaginary parts for some real polynomial $p$ on $\R$, it is clear that there cannot exist a positive Radon measure $\mu$ such that $\widetilde{\mathscr{L}}(p) = \int_\R p(x) \, \rd \mu(x)$.

For the purposes of the rest of this article we work with the latter description of the moment problem formulation. That is, given a real-valued linear functional $\mathscr{L}$ defined on the space of all real polynomials on $\R$ we define the sequence of real numbers $s \equiv (s_n)_{n \in \N_0}$ as 
\begin{equation}\label{eq:sn_definition}
s_n := \mathscr{L}(x^n)\,,\quad  \forall \, n \in \N_0\,.
\end{equation}

Therefore, we can restate the classical one dimensional moment problem in these more specific circumstances as follows: given a real-valued linear functional $\mathscr{L}$ on the space of real polynomials on $\R$ (that additionally defines the sequence of real numbers $s$ in \eqref{eq:sn_definition}), under what conditions does there exist a positive Radon measure $\mu$ on $\R$ such that 
\begin{equation}\label{eq:moment_equation}
\mathscr{L}(x^n) = \int_\R x^n \, \rd \mu \text{ for all } n \in \N_0\,?
\end{equation}
In the case of the existence of such a measure, we say that $\mathscr{L}$ is a \emph{Hamburger moment functional} (or simply \emph{moment functional}), and also we say that the sequence $s$ that it defines by \eqref{eq:sn_definition} is a \emph{Hamburger moment sequence}. Moreover we wish to know, in the case of existence of such a measure, whether it is \emph{unique}--also known as the \emph{determinacy} of the moment problem. Before delving into this argument we state some necessary definitions and existence results for the Hamburger moment problem.

\begin{definition}\label{def:positive_def}
Let $(s_n)_{n \in \N_0}$ be a sequence of real numbers. We say the the sequence is positive semidefinite if for all $(\xi_0,\xi_1,\xi_2,\dots,\xi_n) \in \R^{n+1}$ for $n \in \N_0$ we have
\begin{equation}\label{eq:def_possemidef}
\sum_{i,j = 0}^n s_{i+j}\xi_i\xi_j \geqslant 0\,,
\end{equation}
and we say that the sequence is positive definite if
\begin{equation}\label{eq:def_posdef}
\sum_{i,j = 0}^n s_{i+j}\xi_i\xi_j > 0\,,
\end{equation}
for all $(\xi_0,\xi_1,\xi_2,\dots,\xi_n) \in \R^{n+1}\setminus\{0\}$.
\end{definition}

\begin{theorem}[Hamburger's Theorem, Theorem~3.8 \cite{schmu_moment}]\label{th:Hamburger_solutions}
Let $\mathscr{L}$ be a real-valued linear functional defined on the space of real polynomials on $\R$, and let the sequence $s \equiv (s_n)_{n \in \N_0}$ of real numbers be defined by \eqref{eq:sn_definition}. The following are equivalent:
\begin{itemize}
	\item[(i)] $\mathscr{L}$ is a Hamburger moment functional, i.e., the sequence $s$ is a Hamburger moment sequence, so there exists a positive Radon measure $\mu$ on $\R$ such that $x^n \in L^1(\R,\mu)$ and
		\[\mathscr{L}(x^n) = \int_\R x^n \, \rd \mu\,, \quad \forall \, n \in \N_0\,,\]
	\item[(ii)] the sequence $s$ is positive semidefinite,
	\item[(iii)] the linear functional $\mathscr{L}$ is a positive linear functional on the space of real polynomials on $\R$, i.e., $\mathscr{L}(p^2) \geqslant 0$ for all real polynomials $p$ on $\R$. 
\end{itemize}

\end{theorem}

\begin{remark}\label{rem:determinacy_HMP_selfadjoint}
If we consider any $g \in C^\infty(A)$ for a self-adjoint operator $A$ on Hilbert space $\cH$, we see that the positive (and bounded) Radon measure $\mu_g^{(A)}(\Omega) := \langle E^{(A)}(\Omega)g, g\rangle$ on $\R$ generates a linear functional $\mathscr{L}_\mu$ on the space of all polynomials over $\R$, $\mathscr{L}_\mu(p) := \int_\R p(\lambda) \, \rd \mu_g^{(A)}(\lambda)$ for some polynomial $p$. This linear functional is clearly real-valued for the real polynomials on $\R$.

On the subspace of all real polynomials on $\R$, the functional $\mathscr{L}_\mu$ is a positive linear functional, i.e., for any real polynomial $p$ we have $\mathscr{L}_\mu(p^2) \geqslant 0$. By Theorem~\ref{th:Hamburger_solutions} the sequence of real numbers given by $s_n := \mathscr{L}_\mu(\lambda^n) = \int_\R \lambda^n \, \rd \mu_g^{(A)}(\lambda)$ is positive semidefinite.
\end{remark}

We now move on to the conditions for the \emph{determinacy} of the Hamburger moment problem. Without loss of generality, we suppose that the linear functional $\mathscr{L}$ on the space of real polynomials on $\R$ is real-valued and positive. Therefore, from Theorem~\ref{th:Hamburger_solutions}(i), $\mathscr{L}$ is a Hamburger moment functional and thus there exists a positive Radon measure $\mu$ on $\R$ such that
\[\mathscr{L}(x^n) = \int_{\R} x^n \, \rd \mu \,, \quad \forall \, n \in \N_0\,,\]
and we define the sequence of real numbers $s \equiv (s_n)_{n \in \N_0}$ as in \eqref{eq:sn_definition}. It is clear that $\mathscr{L}$ extends to the space of \emph{all} polynomials on $\R$ by linearity, i.e., for a (not necessarily real) polynomial $p$ on $\R$, $p(x) = \sum_{i = 0}^n c_i x^n$ for $c_i \in \C$ for all $i=0,1,\dots,n$, we may naturally extend $\mathscr{L}$ as follows,
\[\mathscr{L}(p) := \sum_{i=0}^n c_i \mathscr{L}(x^n)=\int_\R p(x) \,\rd \mu\,.\]
(We do not distinguish between the original moment functional and the extended moment functional to avoid unnecessary obfuscation.) Then for the space of all polynomials on $\R$ we have that $\mathscr{L}$ defines a scalar product, anti-linear in the first entry, in the following way.
\begin{equation}
\langle p,q\rangle_s := \mathscr{L}(\bar{p}q)\,,
\end{equation} 
for any $p,q$ polynomials on $\R$. We see that given any non-zero polynomial $p$ on $\R$ of degree $n \in \N_0$
\[\langle p,p\rangle_s = \mathscr{L}(\vert p \vert^2) = \int_\R \overline{p(x)}p(x) \, \rd \mu = \sum_{i,j=0}^n s_{i+j} \bar{c_i}c_j\,,\]
where $p(x) = \sum_{i = 0}^n c_i x^i$ and $c_i \in \C$ for all $i = 0,1,\dots,n$ and $(c_0,c_1,\dots,c_n) \neq 0$. Thus, owing to \eqref{eq:def_posdef} we see that $\langle p,p\rangle_s > 0$. Indeed,
\[\Re \sum_{i,j=0}^n s_{i+j} \bar{c_i}c_j = \sum_{i,j=0}^n s_{i+j} \Re \big ( \bar{c_i}c_j\big )\,,\]
and
\[\Re (\bar{c_i}c_j) = \Re (c_i) \Re (c_j) + \Im (c_i) \Im (c_j)\,,\]
so that
\[\sum_{i,j=0}^n s_{i+j} \bar{c_i}c_j = \sum_{i,j=0}^n s_{i+j} \Re (c_i) \Re (c_j) + \sum_{i,j=0}^n s_{i+j} \Im (c_i) \Im (c_j)\,.\]
The other properties of scalar products are easily checked.

%linear functional $\mathscr{L}$ on the space of real polynomials over $\R$, that defines a sequence of real numbers $s_n := \mathscr{L}(x^n)$ for all $n \in \N_0$. We further suppose that the Hamburger moment problem admits a solution, i.e., there exists a positive Radon measure $\mu$ such that $\mathscr{L}(x^n) = \int_{\R} x^n \, \rd \mu$, and moreover we suppose (without loss of generality) that the moment sequence $s$ given by the moment functional $\mathscr{L}$ is positive definite. We may simply extend $\mathscr{L}$ to the space of \emph{all} polynomials on $\R$ in the following way
%\[\mathscr{L}((\alpha + i\beta) x^n) := \mathscr{L}(\alpha x^n) + i\mathscr{L}(\beta x^n)\,, \quad \forall\, n \in \N_0\,,\]
%where $\alpha, \beta \in \R$. 

Therefore the completion of the space of polynomials on $\R$ equipped with the scalar product $\langle \cdot, \cdot\rangle_s$ and corresponding norm $\Vert \cdot \Vert_{s}^2 := \langle \cdot, \cdot\rangle_s$ gives us a Hilbert space that we call $\cH_s$. We also note that the multiplication operator on the space of all polynomials over $\R$, $p(x) \xmapsto{T_x} xp(x)$ is a symmetric and \emph{densely defined} operator on $\cH_s$. Indeed,
\[\langle xp(x), q(x) \rangle = \mathscr{L}(\overline{xp(x)}q(x)) = \mathscr{L}(\overline{p(x)}xq(x)) = \langle p(x), xq(x)\rangle\,.\]
It is also known that $T_x$ has deficiency indices of either $(0,0)$ or $(1,1)$ \cite[Corollary~6.7]{schmu_moment}.

\begin{theorem}[Theorem~6.10 \cite{schmu_moment}]\label{th:Hamburger_determinancy}
Let $\mathscr{L}$ be a real-valued linear functional on the space of real polynomials on $\R$, and suppose the sequence $s$ defined by \eqref{eq:sn_definition} is positive definite. Then the Hamburger moment problem is determinate if and only if the multiplication operator $T_x$ on $\cH_s$ has deficiency indices $(0,0)$, or in other words, $T_x$ is essentially self-adjoint on the space $\cH_s$. If this holds and $\mu$ is the unique representing measure for $\mathscr{L}$, then the space of polynomials on $\R$ is dense in the space $L^2(\R,\mu)$, i.e., $\cH_s \cong L^2(\R,\mu)$.
\end{theorem}

\begin{corollary}[Corollary~6.11 \cite{schmu_moment}]\label{cor:L2density_Hamburgerdeterminacy}
Suppose that $\mu$ is a positive Radon measure on $\R$ that defines a Hamburger moment functional $\mathscr{L}(p) := \int_\R p(x) \, \rd \mu$ on the real polynomials $p$ on $\R$. Then the moment sequence $s$ defined by \eqref{eq:sn_definition} is determinate, i.e., the Hamburger moment problem is determinate, if and only if the space of all polynomials on $\R$ is dense in $L^2(\R,(1+x^2)\rd\mu)$.
\end{corollary}

\begin{remark}
Under the condition that the sequence $s$ as defined by \eqref{eq:sn_definition} is positive definite, each representing measure $\mu$ for the Hamburger moment problem on $\R$ has infinite support, that is, its support contains an infinite subset of $\R$ (\cite[Proposition~3.11]{schmu_moment}). We note that for the case that $s$ is a positive \emph{semidefinite} sequence that the determinacy of the Hamburger moment problem is under full control. Indeed, in the case that there exists a real polynomial $p(x) = \sum_{i=0}^n \xi_i x^i$, $\xi_i \in \R$ for all $i$, such that the moment functional applied to $p^2$ is zero, i.e., $0 = \sum_{i,j = 0}^n \xi_i \xi_j s_{i+j} = \mathscr{L}(p^2)$, then any positive Radon measure representing $\mathscr{L}$ must have support at exactly finitely many points. In the case that a representing measure $\mu$ for $\mathscr{L}$ is supported on finitely many points, and thus has compact support, it is the only representing measure for the functional $\mathscr{L}$, i.e., the Hamburger moment problem is determinate (\cite[Corollary~4.2]{schmu_moment}).
\end{remark}

In the following proposition we give a sufficient condition that is also practical to check, known as Carleman's Condition, that ensures determinacy of the Hamburger moment problem. The statement and proof may be found in \cite[Theorem~4.3]{schmu_moment}.

\begin{proposition}[Carleman's Condition]\label{prop:CarlemanCondition}
Suppose that the sequence $(s_n)_{n \in \N_0}$ defined by \eqref{eq:sn_definition}, for a real-valued linear functional $\mathscr{L}$ on the real polynomials on $\R$, is a positive semidefinite sequence. If $(s_n)_{n \in \N_0}$ satisfies the Carleman Condition
\begin{equation}\label{eq:CarlemanCondition}
\sum_{n = 0}^\infty s_{2n}^{-\frac{1}{2n}} = +\infty
\end{equation}
then the Hamburger moment problem is determinate.
\end{proposition}

\subsection{Structural Properties of $\Kc{A}{g}$}\label{subsec:selfadjoint_struct_Krylov}~

We now begin our investigation of the structural properties of the Krylov subspace $\Kc{A}{g}$ under the condition that $A$ is a self-adjoint operator on a Hilbert space $\cH$. In order to proceed, we first describe some important classes of vectors, namely the bounded, analytic, and quasi-analytic classes. More information on these classes may be found in \cite[Chapter~7]{schmu_unbdd_sa}.

\begin{definition}\label{def:vector_classes}
Let $A$ be a self-adjoint operator on a Hilbert space $\cH$, and let $g \in C^\infty(A)$. We say that
\begin{itemize}
	\item[(i)] $g$ is of the \emph{bounded class} with respect to $A$, i.e., $g\in\mathcal{D}^{b}(A)$, if there exists some $B_g > 0$ such that $\Vert A^n g\Vert_\cH \leqslant B_g^n$ for all $n \in \N$,
	\item[(ii)] $g$ is of the \emph{analytic class} with respect to $A$, i.e., $g\in\mathcal{D}^{a}(A)$, if there exists some $C_g > 0$ such that $\Vert A^n g\Vert_\cH \leqslant C_g^n n!$,
	\item[(iii)] and $g$ is of the \emph{quasi-analytic class} with respect to $A$, i.e., $g\in\mathcal{D}^{qa}(A)$, if \[\sum_{n = 1}^\infty \Vert A^n g\Vert_\cH^{-\frac{1}{n}} = +\infty\,.\] 
\end{itemize}
Clearly we have that $\mathcal{D}^b(A) \subset \mathcal{D}^a(A) \subset \mathcal{D}^{qa}(A)$. Moreover, $\mathcal{D}^b(A)$ and $\mathcal{D}^a(A)$ are linear subspaces of $\cH$, while $\mathcal{D}^{qa}(A)$ is not necessarily so.
\end{definition}

We note that for $A$ a self-adjoint operator and $g \in C^\infty(A)$ we have a positive bounded Radon measure on the real line $\mu_g^{(A)}(\Omega) := \langle E^{(A)}(\Omega)g, g\rangle$, for $\Omega \subset \R$ a Borel measurable set and $E^{(A)}(\Omega)$ the unique projection valued spectral measure for $A$. Therefore, we can define a linear functional $\mathscr{L}_\mu$ on the space of polynomials, as in Remark~\ref{rem:determinacy_HMP_selfadjoint}, as follows
\begin{equation}\label{eq:Lmu}
\mathscr{L}_\mu(p(\lambda)) := \int_\R p(\lambda) \, \rd \mu_g^{(A)}\,,
\end{equation}
and this gives rise to a positive semidefinite moment sequence $s \equiv (s_n)_{n \in \N_0}$, where
\begin{equation}\label{eq:Lmu_s}
s_n := \mathscr{L}_\mu(\lambda^n)\,,
\end{equation}
for all $n \in \N_0$. It is simple to see that $s$ is positive semidefinite: indeed as $\mu_g^{(A)}$ is a positive bounded measure and $g \in C^\infty(A)$, $\mathscr{L}_\mu(\lambda^n) < +\infty$ for all $n \in \N_0$ and also $\mathscr{L}(p^2(\lambda)) \geqslant 0$ for any real polynomial $p$. Therefore owing to Theorem~\ref{th:Hamburger_solutions} we have that $s$ is positive semidefinite. As the functional $\mathscr{L}_\mu$ is constructed from the positive Radon measure $\mu_g^{(A)}$ on $\R$ and $\lambda^n \in L^1(\R,\mu_g^{(A)})$ we shall concern ourselves only with the \emph{determinacy} problem. In other words, we wish to know whether the measure $\mu_g^{(A)}$ is the only positive Radon measure that solves the Hamburger moment problem. To this end, we list below a sufficient condition for this to be the case.

\begin{proposition}\label{prop:quasianalytic_determinate}
Let $A$ be a self-adjoint operator on a Hilbert space $\cH$ and let $g \in \mathcal{D}^{qa}(A)$. Then the Hamburger moment problem with moment functional as in \eqref{eq:Lmu} is determinate. Moreover, $\Kc{A}{g} \cong L^2\big(\R,\mu_g^{(A)}\big)$.
\end{proposition}

\begin{proof}
Let $\mu_g^{(A)}(\cdot) = \langle E^{(A)}(\cdot)g, g\rangle$ be the positive Radon measure generated by the projection valued spectral measure $E^{(A)}$ for $A$. By the functional calculus
\[\Vert A^n g\Vert_\cH^2 = \int_\R \lambda^{2n} \, \rd \mu_g^{(A)}\,,\]
so that for the Hamburger moment sequence $s \equiv (s_n)_{n \in \N_0}$ defined by $s_n := \int_\R \lambda^n \,\rd \mu_g^{(A)}$ for all $n \in \N_0$, we see that $s_{2n} = \Vert A^n g\Vert_\cH^2$ for all $n \in \N_0$.

Then, as $g$ is quasi-analytic,
\[\sum_{n=1}^\infty s_{2n}^{-\frac{1}{2n}} = \sum_{n=1}^\infty \Vert A^n g\Vert_\cH^{-\frac{1}{n}} = +\infty\,,\]
so that the Carleman Condition is satisfied, and therefore the Hamburger moment problem is determinate.

The proof that $\Kc{A}{g} \cong L^2\big(\R,\mu_g^{(A)}\big)$ is a direct consequence of Theorem~\ref{th:Hamburger_determinancy}.
\end{proof}

\begin{remark}\label{rem:Kapprox_HamburgerDet}
We note that in the above Proposition, we actually have the more general result that $\Kc{A}{g} \cong L^2\big(\R,\mu_g^{(A)}\big)$ \emph{if} the Hamburger moment problem defined by \eqref{eq:Lmu} and \eqref{eq:Lmu_s} is determinate. As in the proof of Proposition~\ref{prop:quasianalytic_determinate} this is a direct consequence of Theorem~\ref{th:Hamburger_determinancy}.

In fact, with regards to Proposition~\ref{prop:quasianalytic_determinate}, it is well known \cite{schmu_moment} that the Carleman Condition is sufficient but not necessary to guarantee the determinacy of the moment problem.
\end{remark}

\begin{remark}\label{rem:KisomorphismL2}
We briefly explain and show the isomorphism between $\Kc{A}{g}$ and $L^2\big(\R,\mu_g^{(A)}\big)$ under the condition of the determinacy of the Hamburger moment problem. The reasoning that follows is merely a reshaping of the elements from the proof of \cite[Theorem~7.2]{CM-2019_ubddKrylov}, as well as \cite[Theorem~4.1]{C-KrylovNormal-2022} for bounded normal operators. We claim that when the Hamburger moment is determinate, there exists the following isomorphism
\begin{equation}\label{eq:KcisomorphismL2}
\begin{split}
L^2\big(\R,\mu_g^{(A)}\big) & \xrightarrow{\cong} \Kc{A}{g}\,,\\
f & \mapsto f(A)g\,,
\end{split}
\end{equation}\label{eq:KcisomorphismL2_funccalc}
where $f(A)$ is understood in terms of the functional calculus
\begin{equation}
f(A) := \int_\R f(\lambda) \, \rd E^{(A)}(\lambda)\,.
\end{equation}
Indeed, given any $f \in L^2\big(\R,\mu_g^{(A)}\big)$, as Theorem~\ref{th:Hamburger_determinancy} guarantees the space of polynomials is dense in $L^2\big(\R,\mu_g^{(A)}\big)$, there exists a sequence $(p_n)_{n \in \N}$ of polynomials such that $p_n \xrightarrow{L^2} f$ as $n \to \infty$. From the functional calculus, we see that
\[\Vert p_n(A) g - f(A) g \Vert_\cH^2 = \int_\R \vert p_n(\lambda) - f(\lambda) \vert^2 \, \rd \mu_g^{(A)} \xrightarrow{n \to \infty} 0\,.\]
Therefore $f(A)g \in \Kc{A}{g}$. Now suppose we have some $u \in \Kc{A}{g}$. Then there exists a sequence of polynomials $(q_n)_{n \in \N}$ such that $(q_n(A)g)_{n \in \N} \subset \Kc{A}{g}$ is a Cauchy sequence and converges strongly in $\cH$ to $u$. The functional calculus implies that the sequence of polynomials $(q_n)_{n \in \N}$ is Cauchy in $L^2\big(\R,\mu_g^{(A)}\big)$, indeed
\[\Vert q_n(A)g - q_m(A)g \Vert^2 = \int_\R \vert q_n(\lambda) - q_m(\lambda) \vert^2 \, \rd \mu_g^{(A)} = \Vert q_n - q_m \Vert_{L^2}\,, \quad \forall\, n,m \in \N\,.\]
Therefore $q_n$ converges to some $h \in L^2\big(\R,\mu_g^{(A)}\big)$, and thus $h(A)g = u$. It is clear that the mapping \eqref{eq:KcisomorphismL2} is a linear and continuous bijection between $L^2\big(\R,\mu_g^{(A)}\big)$ and $\Kc{A}{g}$.

\end{remark}

We now use Corollary~\ref{cor:L2density_Hamburgerdeterminacy} in order to show that the Krylov core condition holds when $g$ is a vector that makes the Hamburger moment problem determinate.

\begin{proposition}\label{prop:Kcore_Hamburgerderteminate}
Let $A$ be a self-adjoint operator on a Hilbert space $\cH$ and let $g \in C^\infty(A)$. If the Hamburger moment problem defined by \eqref{eq:Lmu} and \eqref{eq:Lmu_s} is determinate, then the Krylov core condition is satisfied, i.e., $\Kc{A}{g} \cap \mathcal{D}(A) = \Kc{A}{g}^V$. Moreover, we have the inclusions
\begin{equation}\label{eq:Kcore_Kred}
A \big(\Kc{A}{g}\cap\mathcal{D}(A)\big) \subset \Kc{A}{g}\,, \quad A \big(\K{A}{g}^\perp \cap \mathcal{D}(A)\big) \subset \K{A}{g}^\perp\,.
\end{equation}
\end{proposition}

\begin{proof}
It is already clear that $\Kc{A}{g}^V \subset \Kc{A}{g} \cap \mathcal{D}(A)$ so our proof will show that $\Kc{A}{g}^V \supset \Kc{A}{g} \cap \mathcal{D}(A)$. Indeed, take any vector $v \in \Kc{A}{g} \cap \mathcal{D}(A)$. Then the vector $v$ can be represented as a function $f$ in $L^2\big(\R,\mu_g^{(A)}\big)$ by way of the functional calculus (Remark~\ref{rem:KisomorphismL2})
\[v = f(A)g = \int_\R f(\lambda) \, \rd E^{(A)}g\,.\]
$\Vert Av\Vert_\cH^2 < +\infty$ as $v \in \mathcal{D}(A)$ and therefore
\[\int_\R \lambda^2 \vert f(\lambda) \vert^2 \, \rd \mu_g^{(A)} =\Vert Av\Vert_\cH^2 < +\infty\,,\]
from which $f \in L^2\big(\R,(1+\lambda^2)\mu_g^{(A)}\big)$. 

As the Hamburger moment problem is determinate, by Corollary~\ref{cor:L2density_Hamburgerdeterminacy} we know that the polynomials are dense in $L^2\big(\R,(1+\lambda^2)\mu_g^{(A)}\big)$ and therefore there exists a polynomial approximation $(p_n)_{n \in \N}$ such that $p_n \xrightarrow{n \to \infty} f$ in $L^2\big(\R,(1+\lambda^2)\mu_g^{(A)}\big)$. Thus
\begin{align*}
\Vert p_n(A)g - f(A)g \Vert_{A}^2 & = \Vert p_n(A)g - f(A)g\Vert_\cH^2 + \Vert Ap_n(A)g - Af(A)g\Vert_\cH^2 \\
 & = \int_\R (1 + \lambda^2)\vert p_n(\lambda) - f(\lambda) \vert^2 \, \rd \mu_g^{(A)} \xrightarrow{n \to \infty} 0\,,
\end{align*}
and therefore $v = f(A)g \in \Kc{A}{g}^V$. By Propositions~\ref{prop:invarianceunderadjoint} and \ref{prop:invariance} we have the inclusions \eqref{eq:Kcore_Kred}.
\end{proof}

\begin{remark}\label{rem:prop5.4}
A proof that the class of bounded vectors satisfies the Krylov core condition is presented in \cite[Theorem~7.1]{CM-2019_ubddKrylov} and also in \cite[Theorem~4.4]{CM-book-Krylov-2022}. These two proofs were done in an explicit way using a regularising sequence of vectors (\cite[Theorem~7.1]{CM-2019_ubddKrylov}) and using a purely measure-theoretic approach (\cite[Theorem~4.4]{CM-book-Krylov-2022}). The proof of Proposition~\ref{prop:Kcore_Hamburgerderteminate} is very direct and straightforward in comparison to the previous two studies, yet it requires the non-trivial knowledge of results in the theory of the moment problem in one dimension.
\end{remark}

\begin{remark}\label{rem:open_questions}
At this point we remark on the two open questions posed in \cite{CM-2019_ubddKrylov}. These were
\begin{itemize}
	\item[(Q1)] When $A$ is self-adjoint and $g \in C^\infty(A)$, is it true that $\Kc{A}{g}^V = \Kc{A}{g}\cap\mathcal{D}(A)$?
	\item[(Q2)] When $A$ is self-adjoint and $g \in C^\infty(A)$ is it true that $A \big( \Kc{A}{g} \cap \mathcal{D}(A)\big) \subset \Kc{A}{g}$?
\end{itemize}
We note that in \cite{CM-2019_ubddKrylov} a partial answer in the affirmative for both (Q1) and (Q2) was given in Theorem~7.1 therein, namely for the case in which $g \in \mathcal{D}^b(A)$. Our Proposition~\ref{prop:Kcore_Hamburgerderteminate} expands on this answer by showing that the entire class of smooth vectors that result in a determinate Hamburger moment problem provide a positive answer to these two questions.
\end{remark}

We shall now prove a stronger result than \eqref{eq:Kcore_Kred} for the class of bounded vectors, that will then be extended to larger classes by way of a perturbative analysis in Section \ref{sec:selfadjoint_perturbations}. Before we do so, we shall state an important approximation theorem for continuous functions on $\R$ called Carleman's Theorem (not to be confused with Proposition~\ref{prop:CarlemanCondition}, namely Carleman's Condition for the determinacy of the Hamburger moment problem).

\begin{theorem}[Chapter~IV, Section~3, Theorem~1 \cite{Gaier_ComplexApproximation}]\label{th:Carleman}
Let $f:\R \to \C$ be a continuous function. Given any strictly positive function $\varepsilon:\R \to [0,\infty)$, there exists an entire function $h:\C \to \C$ such that 
\[\vert h(x) - f(x) \vert < \varepsilon(x)\,,\quad \forall\,x \in\R\,.\]
\end{theorem}

Calculating the range of separations (Definition~\ref{def:subspace_dist2}) between $\Kc{A}{g}$ and \\ $A\big(\K{A}{g}^{\perpV}\big)$ gives us an indicator as to whether $\Kint{A}{g} = \{0\}$. Indeed, take any vector $v \in \K{A}{g}^{\perpV}$ and $u \in \K{A}{g}$ such that $\Vert u\Vert_\cH = 1$ and $\Vert Av\Vert_\cH = 1$. Then
\[\Vert u - Av\Vert_\cH^2 = 2\big(1 - \Re \langle u,Av\rangle\big)\,,\]
and
\[\langle u, Av\rangle = \int_{\R} p(\lambda)\lambda\,\rd \bigl \langle E^{(A)}(\lambda)g, v \bigr \rangle\,,\]
where $p$ is the polynomial such that $u = \int_{\R} p(\lambda) \,\rd E^{(A)}(\lambda)g = p(A)g$. Our aim is to show that $\langle u, Av \rangle=0$ under certain circumstances. We now state and prove a proposition central to our structural investigations of the Krylov subspace structure that indeed shows that $\langle u, Av\rangle$ is zero when $g$ is of the bounded class of vectors.

\begin{proposition}\label{prop:bddvec_invariances}
Let $A$ be a self-adjoint operator on a Hilbert space $\cH$ and let $g \in \mathcal{D}^b(A)$. Then $\K{A}{g}^{\perpV} \subset \K{A}{g}^\perp$ and $A\big(\K{A}{g}^{\perpV}\big) \subset \K{A}{g}^\perp$.
\end{proposition}
\begin{proof}
Given any $u \in \K{A}{g}$ and any $v \in \K{A}{g}^{\perpV}$,
\[0 = \langle u,v \rangle_V= \int_{\R} p(\lambda)(1 + \lambda^2) \,\rd \bigl \langle E^{(A)}(\lambda)g,v\bigr\rangle\,,\]
where $u = \int_\R p(\lambda) \,\rd E^{(A)}g$ for a polynomial $p$ that corresponds to the vector $u$. Owing to the invariance of $\K{A}{g}$ under $A$, we also have that
\[0 = \langle A^k u,v \rangle_V = \langle A^k p(A)g,v \rangle_V=  \int_{\R} p(\lambda)(1 + \lambda^2)\lambda^k \, \rd \bigl \langle E^{(A)}(\lambda)g, v\bigr\rangle\,,\]
for any polynomial $p$ and for all $k \in \N_0$. We will construct a sequence of polynomials $(p_m)_{m\in\N}$ such that we have the limit
\begin{equation}\label{eq:star_scalar}
0 = \langle A^k p_m(A)g,v \rangle_V \xrightarrow{m \to \infty} \langle A^k g, v \rangle\,, \tag{*}
\end{equation}
which is tantamount to the convergence of the integrals
\begin{equation}\label{eq:star_integral}
\int_{\R} p_m(\lambda)(1 + \lambda^2)\lambda^k \, \rd \bigl\langle E^{(A)}(\lambda)g,v\bigr\rangle \xrightarrow{m \to \infty} \int_{\R}\lambda^k \, \rd \bigl\langle E^{(A)}(\lambda)g,v\bigr\rangle\,, \tag{**}
\end{equation}
which then implies that $A^k g \perp v$, for $k = 0,1$ and for any given $v \in \K{A}{g}^{\perpV}$. Then for all $k \in \N_0$ we will show that $\langle A^kg, v\rangle=0$. 

To achieve this, we invoke Carleman's Theorem (Theorem~\ref{th:Carleman}), noting that $(1 + \lambda^2)^{-1}$ is a continuous function on $\R$. We choose an $\varepsilon > 0$ and let the positive error function be $\varepsilon(\lambda) = \varepsilon$. Therefore there exists an entire function $h_\varepsilon : \C \to \C$ such that $\vert h_\varepsilon(\lambda) - (1 + \lambda^2)^{-1}\vert < \varepsilon$ for all $\lambda \in \R$. Let us define a polynomial sequence as follows
\[q_n^{(h_\varepsilon)}(\lambda) := \sum_{i = 0}^n \frac{h_\varepsilon^{(i)}(0)}{i!} \lambda^i\,,\quad n \in \N_0\,,\]
which is the $n$-th order truncation of the Maclaurin series for $h_\varepsilon(\lambda)$. As $h_\varepsilon$ is entire, the series has an infinite radius of convergence, so $q_n^{(h_\varepsilon)} \xrightarrow{n \to \infty} h_\varepsilon$ uniformly on compact subsets of $\C$. The function
\[f(\lambda) := \sum_{i = 0}^\infty \frac{\vert h_\varepsilon^{(i)}(0)\vert}{i!} \vert \lambda\vert^i\]
dominates both $h_\varepsilon$ and each $q_n^{(h_\varepsilon)}$, so that $2f(\lambda)(1+\lambda^2)\vert\lambda\vert^k$ dominates the function $\big\vert \big(q_n^{(h_\varepsilon)}(\lambda) - h_\varepsilon(\lambda)\big)(1 + \lambda^2)\lambda^k \big\vert$ for all $k \in \N_0$. (Note that while $f$ depends on $h_\varepsilon$, we leave out the dependence in its notation in order to avoid obfuscation.)

%It is clear that the function $\lambda \mapsto (1+\lambda^2)^{-1}$ is in $L^2(\R,\mu_g^{(A)})$. As the Hamburger moment problem is definite, we know that the polynomials are dense in $L^2(\R,\mu_g^{(A)})$ (Theorem~\ref{th:Hamburger_determinancy}). We claim that the polynomials are also dense in $L^2(\R,(1+\lambda^2)\mu_g^{(A)})$. Indeed, let $h \in L^2(\R,(1+\lambda^2)\mu_g^{(A)})$. It is clear that $h(\lambda)(1+\lambda^2)^{-\frac{1}{2}} \in L^2(\R,\mu_g^{(A)})$ and therefore there exists a sequence of polynomials $(q_n)_{n \in \N}$ such that $q_n(\lambda) \xrightarrow{L^2(\R,\mu_g^{(A)})} h(\lambda)(1+\lambda^2)^{-\frac{1}{2}}$ as $n \to \infty$. Thus
%\[\int_\R \vert q_n(\lambda) - h(\lambda)(1+\lambda^2)^{-\frac{1}{2}}\vert^2 (1+\lambda^2)\,\rd \mu_g^{(A)}\]
%
%
%
%
%
%
%
%
%
%
%
%
%
%
% and therefore there exists a sequence of polynomials $(p_n)_{n \in \N}$ such that $p_n(\lambda) \xrightarrow{L^2} (1+\lambda^2)^{-1}$. 
%
%Using \cite[Lem. 4.8(ii)]{schmu_unbdd_sa} we obtain the following inequalities for $k = 0$ or $1$
%\begin{align*}
%\bigg\vert \int_{\R} (p_n(\lambda) - (1+\lambda^2)^{-1})&(1 + \lambda^2) \lambda^k \, \rd \scalar{E^{(A)}(\lambda)g}{v} \bigg\vert^2 \\
% & \leqslant \int_{\R} \vert (p_n(\lambda) - (1+\lambda^2)^{-1})(1 + \lambda^2)\vert^2 \, \rd \mu_g^{(A)} \int_\R \lambda^{2k} \, \rd \mu_v^{(A)}
%\end{align*}

To show that $f(\lambda)$ actually exists for each $\lambda \in \R$ we consider the Cauchy estimate for $h_\varepsilon^{(i)}(0)$, namely
\[\vert h_\varepsilon^{(i)}(0) \vert \leqslant \frac{M_R i!}{R^i}\,,\]
where $R > 0$ and $M_R = \max_{z = R} \vert h_\varepsilon (z) \vert$. Then for fixed $\lambda \in \R$
\[\frac{\vert h_\varepsilon^{(i)}(0)\vert}{i!} \vert \lambda\vert^i \leqslant M_R  \frac{\vert \lambda \vert^i}{R^i}\,,\]
and thus the summation converges as we may choose $R > \vert \lambda \vert$ without restriction.

We claim that $f(\lambda)(1 + \lambda^2)\vert \lambda \vert^k \in L^2\big(\R,\mu_g^{(A)}\big)$ for all $k \in \N_0$. Indeed, consider $f^2(\lambda)$
\[\vert f(\lambda) \vert^2 = f^2(\lambda) = \sum_{i = 0}^\infty \frac{\vert h_\varepsilon^{(i)}(0)\vert}{i!} \vert \lambda\vert^i \sum_{j = 0}^\infty \frac{\vert h_\varepsilon^{(j)}(0)\vert}{j!} \vert \lambda\vert^j\,,\]
so that owing to the (absolute) convergence of both series at each $\lambda \in \R$ we can replace the double summation with the Cauchy product
\[f^2(\lambda) = \sum_{i = 0}^\infty c_i \vert \lambda \vert^i\,, \quad c_i = \sum_{l =0}^i \frac{\vert h_\varepsilon^{(l)}(0)\vert}{l!}\frac{\vert h_\varepsilon^{(i-l)}(0)\vert}{(i-l)!}\,.\]
Taking the Cauchy estimate for some $R > 0$ (to be chosen later), we see that
\[c_i \leqslant M_R^2 \sum_{l = 0}^i \frac{1}{R^i} = \frac{(i + 1) M_R^2}{R^i}\,,\]
and therefore
\[f^2(\lambda) \leqslant M_R^2 \sum_{i = 0}^\infty \frac{i+1}{R^i} \vert \lambda \vert^i\,.\]
Indeed we see that the right side converges for each $\lambda \in \R$ by choosing some $R$ large enough e.g., for fixed $\lambda$ one may choose $R > 4\vert \lambda \vert$ noting that $i+1 < 2^i$ for all $i \in \N$ and $2\vert \lambda \vert \geqslant \vert \lambda \vert$ for all $\lambda \in \R$.

Putting this into the spectral integral, we have by the Monotone Convergence Theorem and the H\"{o}lder inequality
\begin{align*}
\int_\R f^2(\lambda) (1 + & \lambda^2)^2  \lambda^{2k}  \,\rd \mu_g^{(A)}  \\ 
 & \leqslant M_R^2 \sum_{i = 0}^\infty \int_\R \frac{i+1}{R^i}\vert \lambda \vert^i (1 + \lambda^2)^2 \lambda^{2k} \, \rd \mu_g^{(A)} \\
& \leqslant  M_R^2 \sum_{i = 0}^\infty \frac{i+1}{R^i} \left( \int_\R \lambda^{2i} \, \rd \mu_g^{(A)} \right)^{\frac{1}{2}} \left( \int_\R (1 + \lambda^{2})^4 \lambda^{4k} \, \rd \mu_g^{(A)} \right)^{\frac{1}{2}}\\
 & \leqslant \tilde{\mathscr{C}}_k M_R^2 \sum_{i = 0}^\infty \frac{i+1}{R^i} B_g^{i}\,,
\end{align*}
where $\tilde{\mathscr{C}}_k = \big( \int_\R (1 + \lambda^{2})^4 \lambda^{4k} \, \rd \mu_g^{(A)} \big )^{\frac{1}{2}} < \infty$ for all $k \in \N_0$, so that the series on the right side of the last inequality converges for any choice of $R > 4\max\{1, B_g\}$. This proves that $f(\lambda)(1 + \lambda^2)\vert \lambda \vert^k \in L^2\big(\R,\mu_g^{(A)}\big)$ for all $k \in \N_0$.

Thus considering $k=0,1$
\begin{equation}\label{eq:integralA}
\begin{split}
\bigg\vert \int_{\R} \big(q^{(h_\varepsilon)}_n(\lambda) - & h_\varepsilon(\lambda)\big)(1 +  \lambda^2) \, \rd \bigl\langle E^{(A)}(\lambda)g, v\bigr\rangle \bigg\vert^2  \\ 
 & \leqslant \int_{\R} \big \vert q^{(h_\varepsilon)}_n(\lambda) - h_\varepsilon(\lambda)\big\vert^2 (1 + \lambda^2)^2 \, \rd \mu_g^{(A)} \int_{\R} \, \rd \mu_v^{(A)} \\
 & = \Vert v\Vert_\cH^2 \int_{\R} \big\vert q_n^{(h_\varepsilon)}(\lambda) - h_\varepsilon(\lambda)\big\vert^2 (1 + \lambda^2)^2 \, \rd \mu_g^{(A)}\,, 
\end{split}\tag{a}
\end{equation}
and
\begin{equation}\label{eq:integralB}
\begin{split}
\bigg\vert \int_{\R} \big(q_n^{(h_\varepsilon)}(\lambda) - & h_\varepsilon(\lambda)\big)  (1 + \lambda^2) \lambda \,  \rd \bigl \langle E^{(A)}(\lambda)g, v\bigr \rangle \bigg\vert^2  \\ 
 & \leqslant \int_{\R} \big\vert q_n^{(h_\varepsilon)}(\lambda) - h_\varepsilon(\lambda)\big\vert^2 (1 + \lambda^2)^2 \, \rd \mu_g^{(A)} \int_{\R} \lambda^2 \, \rd \mu_v^{(A)} \\
 & = \Vert Av\Vert_\cH^2 \int_{\R} \big\vert q_n^{(h_\varepsilon)}(\lambda) - h_\varepsilon(\lambda)\big\vert^2 (1 + \lambda^2)^2 \, \rd \mu_g^{(A)}\,,
\end{split}\tag{b}
\end{equation}
where we have used \cite[Lemma~4.8(ii)]{schmu_unbdd_sa}. Therefore, the convergence of $\big(q_n^{(h_\varepsilon)} - h_\varepsilon\big)(1+\lambda^2) \xrightarrow{n\to\infty} 0$ in $L^2\big(\R,\mu_g^{(A)}\big)$ implies convergence to zero of the left sides of the inequalities \eqref{eq:integralA} and \eqref{eq:integralB}.

As $\big\vert \big(q_n^{(h_\varepsilon)}(\lambda) - h_\varepsilon(\lambda)\big)(1 + \lambda^2) \big\vert \leqslant 2f(\lambda)(1+\lambda^2) \in L^2\big(\R,\mu_g^{(A)}\big)$ and $q_n^{(h_\varepsilon)}(\lambda)(1 + \lambda^2) \xrightarrow{n \to \infty} h_\varepsilon(\lambda)(1 + \lambda^2)$ point-wise, Lebesgue Dominated Convergence gives
\[ \int_{\R} \big\vert q_n^{(h_\varepsilon)}(\lambda) - h_\varepsilon(\lambda)\big\vert^2 (1 + \lambda^2)^2 \, \rd \mu_g^{(A)} \xrightarrow{n\to\infty} 0\,,\]
and therefore we have the quantities on the left sides of \eqref{eq:integralA} and \eqref{eq:integralB} converge to $0$ as $n \to \infty$.

For $k = 0,1$ there exists an $n_\varepsilon \in \N$ such that 
\begin{align*}
\bigg\vert \int_{\R} q_{n_\varepsilon}^{(h_\varepsilon)}(\lambda)(1 + \lambda^2) \lambda^k \, & \rd \bigl\langle E^{(A)}(\lambda)g,v\bigr\rangle - \\
 & \int_{\R} h_\varepsilon(\lambda)(1 + \lambda^2) \lambda^k \, \rd \bigl\langle E^{(A)}(\lambda)g,v\bigr\rangle \bigg\vert < \varepsilon\,.
\end{align*}
Therefore, for $k = 0,1$,
\begin{align*}
\bigg\vert \int_{\R} \bigg( & q_{n_\varepsilon}^{(h_\varepsilon)}(\lambda) - \frac{1}{1 + \lambda^2}\bigg) (1 + \lambda^2)  \lambda^k \, \rd \bigl\langle E^{(A)}(\lambda)g,v\bigr\rangle\bigg\vert \\
& \leqslant \left\vert \int_{\R} \big(q_{n_\varepsilon}^{(h_\varepsilon)}(\lambda) - h_\varepsilon(\lambda)\big)(1 + \lambda^2) \lambda^k \, \rd \bigl\langle E^{(A)}(\lambda)g,v\bigr\rangle \right\vert \\
 & \qquad\qquad + \left\vert \int_{\R}  \left( h_\varepsilon(\lambda) - \frac{1}{1 + \lambda^2} \right)(1 + \lambda^2) \lambda^k \, \rd \bigl\langle E^{(A)}(\lambda)g, v\bigr\rangle \right\vert \\
 & < \varepsilon + \int_{\R} \left\vert  h_\varepsilon(\lambda) - \frac{1}{1 + \lambda^2} \right\vert(1 + \lambda^2) \vert \lambda^k \vert \, \rd \bigl\langle E^{(A)}(\lambda)g,v\bigr\rangle \\
 & \leqslant \varepsilon + \bigg(\int_{\R} \bigg\vert  h_\varepsilon(\lambda) - \frac{1}{1 + \lambda^2} \bigg\vert^2 (1 + \lambda^2)^2  \, \rd \mu_g^{(A)} \bigg)^{\frac{1}{2}} \left(\int_{\R} \lambda^{2k} \, \rd \mu_v^{(A)}\right)^{\frac{1}{2}} \\
  & \leqslant \varepsilon + \mathscr{C}_k \varepsilon \Vert (\1 + A^2)g\Vert_\cH
\end{align*}
where $\mathscr{C}_0 = \Vert v\Vert_\cH$ and $\mathscr{C}_1 = \Vert Av\Vert_\cH$, and we have used \cite[Lemma~4.8(ii)]{schmu_unbdd_sa}.

We may set $\varepsilon = \frac{1}{m}$ for $m \in \N$, and extract a sequence $(p_m)_{m\in\N}$ with the polynomials $p_m = q_{n_{1/m}}^{(h_{1/m})}$ for all $m \in \N$. We note that in the above construction, the choice of $(p_m)_{m \in \N}$ is independent of the choice of vector $v \in \K{A}{g}^{\perpV}$, and therefore \eqref{eq:star_integral}, and consequently \eqref{eq:star_scalar}, are satisfied for $k = 0,1$ and for all $v \in \K{A}{g}^{\perpV}$. 

We now turn our attention to the case for which $k > 1$. The consequence of the limit in \eqref{eq:star_scalar} for $k=0,1$ is that for all $v \in \K{A}{g}^{\perpV}$
\begin{equation}\label{eq:star_scalark01}
\langle g,v\rangle = 0\,\quad \text{and}\quad 0 = \langle Ag, v\rangle = \langle g, Av\rangle\,.\tag{***}
\end{equation}
For any $v \in \K{A}{g}^{\perpV}$, we have
\[0 = \langle g,v \rangle_V \implies \langle g, v\rangle= - \langle Ag, Av\rangle = - \langle A^2g,v\rangle\,,\]
and
\[0 = \langle Ag,v\rangle_V \implies \langle Ag,v\rangle= - \langle A^2g, Av\rangle = - \langle A^3g,v\rangle\,,\]
so that $\langle A^2g, v\rangle = \langle A^3g, v\rangle = 0$ from \eqref{eq:star_scalark01}. Then,
\[0 = \langle A^2g,v\rangle_V \implies 0 = \langle A^2g, v\rangle= - \langle A^3g, Av\rangle = - \langle A^4g, v\rangle\,,\]
and
\[0 = \langle A^3g,v \rangle_V \implies 0 = \langle A^3g, v\rangle= - \langle A^4g, Av\rangle = - \langle A^5g, v\rangle\,.\]
Continuing in this way, we eventually get that $\langle A^kg, v\rangle = 0$ and $\langle A^kg, Av\rangle = 0$ for all $k \in \N_0$ and for all $v \in \K{A}{g}^{\perpV}$. 

Therefore, both $\K{A}{g}^{\perpV} \subset \K{A}{g}^\perp$ and also $A\big(\K{A}{g}^{\perpV}\big) \subset \K{A}{g}^\perp$.
\end{proof}

As such, for any vector $u \in \Kc{A}{g}$ such that $\Vert u\Vert_\cH = 1$ and any vector $v \in \K{A}{g}^{\perpV}$ such that $\Vert Av\Vert_\cH=1$, we have
\[\Vert u - Av\Vert_\cH^2 = 2\big(1 - \Re \langle u, Av\rangle\big) = 2\]
so that $\mathcal{S}(\K{A}{g},A\big(\K{A}{g}^{\perpV})\big) = \{\sqrt{2}\}$ and therefore, from Theorem~\ref{th:trivialKint_distances} we have that $\Kint{A}{g} = \{0\}$. This is the statement of the next theorem.

\begin{theorem}\label{th:gbdd_Ksolv_Kinttrivial}
Let $A:\cH \to \cH$ be a self-adjoint operator, and $g \in \cH$ be a bounded vector with respect to $A$. Then $\Kint{A}{g} = \{0\}$. Moreover, if $g \in \ran A$ then the inverse linear problem $Af = g$ has a solution $f \in \Kc{A}{g}^V$.
\end{theorem}

\begin{proof}
The Krylov solvability is a consequence of Theorem~\ref{th:KrylovInt}.
\end{proof}

\begin{remark}\label{rem:gbdd_Ksolv_Kinttrivial}
We comment on Theorem~\ref{th:gbdd_Ksolv_Kinttrivial} in light of the previous result \cite[Theorem~4.1]{CM-2019_ubddKrylov}, which for $A$ self-adjoint guarantees the existence of a solution $f$ to the linear inverse problem \eqref{eq:lininv} in $\Kc{A}{g}$. We note that here Theorem~\ref{th:gbdd_Ksolv_Kinttrivial} guarantees the existence of a Krylov solution to \eqref{eq:lininv} in the graph norm closure of $\K{A}{g}$ under the slightly more restrictive condition that $g$ be a bounded vector. As already stated in Remark~\ref{rem:graphnormclosure_advantage}, the existence of Krylov solutions in the graph norm closure of $\K{A}{g}$ is of practical significance. Moreover, the preceding Proposition~\ref{prop:bddvec_invariances} upon which Theorem~\ref{th:gbdd_Ksolv_Kinttrivial} is based, demonstrates important structural properties of $\Kc{A}{g}^V$ that are exploited further in Section~\ref{sec:selfadjoint_perturbations} to expand further Theorem~\ref{th:gbdd_Ksolv_Kinttrivial}, eventually culminating in Theorem~\ref{th:Kreduced}. Also, as noted in Remark~\ref{rem:comparison_unbddpaper}, the formulation of the Krylov intersection used in this study, particularly in proving Theorems~\ref{th:gbdd_Ksolv_Kinttrivial} and \ref{th:Kreduced}, avoids the drawbacks that arise in \cite{CM-2019_ubddKrylov}.

We further note that the previous analysis presented in \cite{CM-Nemi-unbdd-2019} established the existence of a Krylov solution to \eqref{eq:lininv} in $\Kc{A}{g}^V$ under the conditions that $A$ be self-adjoint and positive, and $g \in \mathcal{D}^{qa}(A)$. Yet, it did not contain information regarding the structure of $\K{A}{g}$ and how it interplays with the notion of Krylov solvability.
\end{remark}

\section{Perturbative Analysis for Unbounded Self-Adjoint Operators and Krylov Inner Approximations}\label{sec:selfadjoint_perturbations}
We are now left with the question: what happens to the Krylov subspace structure when one considers $g$ as a more general vector (e.g., analytic as opposed to bounded)? Specifically, we wish to know what happens to the Krylov intersection \eqref{eq:KrylovInt} for more general vectors, and we shall investigate this by using the bounded class of vectors in a perturbation analysis.

\begin{lemma}\label{lem:boundedcore}
Let $A:\cH \to \cH$ be a self-adjoint linear operator. Then the class of bounded vectors $\mathcal{D}^b(A)$ forms a core of $A$, and moreover given any $f \in \mathcal{D}(A)$ the sequence of bounded vectors $(\chi_{[-n,n]}(A)f)_{n \in \N}$ converges in graph norm to $f$.
\end{lemma}
\begin{proof}
Let $f \in \mathcal{D}(A)$, and consider the sequence of vectors $(\chi_{[-n,n]}(A)f)_{n \in \N}$. We clearly know that each of these vectors $f_n := \chi_{[-n,n]}(A)f$ is a bounded vector as
\[\Vert A^k f_n\Vert_\cH^2 = \int_{\R} \lambda^{2k} \chi_{[-n,n]}(\lambda) \, \rd \mu_f^{(A)} \leqslant n^{2k} \Vert f\Vert_\cH^2\,.\]
We also show that $f_n \xrightarrow{\Vert\cdot\Vert_{A}} f$ as $n \to \infty$. Indeed,
\[\Vert f - f_n\Vert_{A}^2 = \int_{\R} (1 + \lambda^2)\vert 1 - \chi_{[-n,n]}(\lambda) \vert^2 \, \rd \mu_f^{(A)} \xrightarrow{n \to \infty} 0\,,\]
from the Lebesgue Dominated Convergence theorem.
\end{proof}

\begin{lemma}\label{lem:p(A)gn_to_p(A)g}
Let $A:\cH\to\cH$ be a self-adjoint operator on Hilbert space $\cH$, let $f \in C^\infty(A)$, and consider the sequence of vectors $(f_n)_{n \in \N}$ defined by $f_n = \chi_{[-n,n]}(A)f$ for all $n \in \N$. Then for any polynomial $p$, we have that $p(A)f_n \xrightarrow[n \to \infty]{\Vert \cdot\Vert_{A}} p(A)f$. Moreover, $\Vert p(A)f_n\Vert_\cH \leqslant \Vert p(A) f\Vert_\cH$ and $\Vert p(A)f_n\Vert_{A} \leqslant \Vert p(A)f\Vert_{A}$ for all $n \in \N$.
\end{lemma}
\begin{proof}
\[\Vert p(A)(f - f_n)\Vert_{A}^2 = \int_{\R} (1 + \lambda^2) \vert p(\lambda) \vert^2 \vert 1 - \chi_{[-n,n]}(\lambda) \vert ^2 \, \rd \mu_f^{(A)}\,.\]
The integral vanishes as $n \to \infty$ due to the Lebesgue Dominated Convergence Theorem. Indeed, $(1 + \lambda^2) \vert p(\lambda) \vert^2$ is a suitable dominating function and the integrand vanishes point-wise. The final inequality is obvious from spectral integrals and the relation $\chi_{[-n,n]}(\lambda) \leqslant 1$ for all $\lambda \in \R$.
\end{proof}

%\begin{lemma}\label{lem:L2quasianalytic}
%Let $A$ be a self-adjoint operator on a Hilbert space $\cH$ and let $g$ be in the class of quasi-analytic vectors
%\begin{equation}\label{eq:lemL2quasianalytic}
%\mathcal{D}^{qa} := \bigg\{x \in C^\infty(A) \,\bigg|\, \sum_{k \in \N} \Hnorm{A^k x}^{-\frac{1}{k}} = +\infty\bigg\}\,.
%\end{equation}
%Then $\Kc{A}{g} \cong L^2(\R,\mu_g^{(A)})$.
%\end{lemma}
%
%\begin{proof}
%Let us assume, by contradiction, that $\Kc{A}{g} \ncong L^2(\R,\mu_g^{(A)})$. Therefore there exists some function $f \in L^2(\R,\mu_g^{(A)})$ such that $f(A)g \perp \Kc{A}{g}$ and $f(A)g \neq 0$. 
%
%Let us define the measure $\nu(\Omega) := \mu_g^{(A)}(\Omega) + \scalar{E^{(A)}(\Omega)f(A)g}{f(A)g}$ for any $\Omega \subset \R$ a Borel measurable set. We note that for any polynomial $p$, $p(A)g \perp f(A)g$ and therefore
%\[0 = \scalar{p(A)g}{f(A)g} = \int_\R p(\lambda)\overline{f(\lambda)} \, \rd \mu_g^{(A)} \]
%
%\end{proof}

Given $A$ a self-adjoint operator on a Hilbert space $\cH$ and $g \in C^\infty(A)$ we define the following subspace
\begin{equation}\label{eq:frakK}
\mathfrak{L}(A,g) := \bigl\{h(A)g \,|\, h \in L^2\big(\R,\mu_g^{(A)}\big)\bigr\}\,.
\end{equation}
It is obvious that $\Kc{A}{g} \subset \mathfrak{L}(A,g)$, and clearly $\mathfrak{L}(A,g)$ is a closed subspace of $\cH$. Indeed, consider $(v_n)_{n \in \N} \subset \mathfrak{L}(A,g)$ a convergent sequence in $\cH$, i.e., $v_n \xrightarrow{n \to \infty} v$, and so there exists for each $n \in \N$ some $h_n \in L^2\big(\R,\mu_g^{(A)}\big)$ such that $v_n = h_n(A)g$. Using the Cauchyness of the sequence $(v_n)_{n \in \N}$ and the functional calculus together show that $(h_n)_{n \in \N}$ is a Cauchy sequence in $L^2\big(\R,\mu_g^{(A)}\big)$, and therefore $h_n \xrightarrow{L^2} h \in L^2\big(\R,\mu_g^{(A)}\big)$ as $n \to \infty$. This implies that $v = h(A)g$ and therefore $v \in \mathfrak{L}(A,g)$. 

In a similar way we also show that the subspace $\mathfrak{L}(A,g) \cap \mathcal{D}(A) \subset V$ is closed in $V$. Indeed, consider $(v_n)_{n \in \N} \subset \mathfrak{L}(A,g) \cap \mathcal{D}(A)$ a convergent sequence in $V$ so that $v_n \xrightarrow{\Vert \cdot\Vert_{A}} v \in V$ as $n \to \infty$. For each $n \in \N$ there exists $h_n \in L^2\big(\R,\mu_g^{(A)}\big)$ such that $v_n = h_n(A)g$, and so $h_n(A)g \xrightarrow{n \to\infty} v$ in $V$. From the convergence in $V$ we have that $h_n(A)g \xrightarrow{n\to\infty} v$ and $Ah_n(A)g \xrightarrow{n\to\infty} Av$ in $\cH$. The closedness of $\mathfrak{L}(A,g)$ in $\cH$ implies that $v \in \mathfrak{L}(A,g)$, and so there exists $h \in L^2\big(\R,\mu_g^{(A)}\big)$ such that $h(A)g = v$. Therefore, $v \in \mathfrak{L}(A,g) \cap \mathcal{D}(A)$.

\begin{proposition}\label{prop:L2isomorphism_frakL}
Let $A$ be a self-adjoint operator on a Hilbert space $\cH$ and let $g \in C^\infty(A)$. Then
\begin{itemize}
	\item[(i)] $\mathfrak{L}(A,g) \cong L^2\big(\R,\mu_g^{(A)}\big)$ in $\cH$,
	\item[(ii)] $\mathfrak{L}(A,g) \cap \mathcal{D}(A) \cong L^2\big(\R,(1+\lambda^2)\mu_g^{(A)}\big)$ in $V$. 
\end{itemize}
\end{proposition}

\begin{proof}
We begin by proving part (i). Let $J:\mathfrak{L}(A,g) \to L^2\big(\R,\mu_g^{(A)}\big)$ be the linear map $h(A)g \mapsto h(\lambda)$. Clearly $J$ is bijective, and moreover $\Vert h(A)g\Vert_\cH^2 = \int_\R \vert h(\lambda) \vert^2 \, \rd \mu_g^{(A)} = \Vert h\Vert_{L^2}^2$. Therefore $J$ is an isometric isomorphism between $\mathfrak{L}(A,g)$ and $L^2\big(\R,\mu_g^{(A)}\big)$.

We now prove part (ii). Let $\tilde{J}: \mathfrak{L}(A,g) \cap \mathcal{D}(A) \subset V \to L^2\big(\R,(1+\lambda^2)\mu_g^{(A)}\big)$ be the linear map $h(A)g \mapsto h(\lambda)$. We first note that $\tilde{J}$ is a surjection. Indeed, let $h \in L^2\big(\R,(1+\lambda^2)\mu_g^{(A)}\big)$, then $h(A)g \in \mathfrak{L}(A,g)$ and $Ah(A)g \in \mathfrak{L}(A,g)$ by the functional calculus, and so $h(A)g \in \mathcal{D}(A)$. We also see that $\Vert h(A)g\Vert_{A}^2 = \int_\R \vert h(\lambda) \vert^2 (1+\lambda^2)\, \rd \mu_g^{(A)} = \Vert h\Vert_{L^2}^2$. So, $\tilde{J}$ is a surjective linear isometry, and therefore the spaces $\mathfrak{L}(A,g)\cap\mathcal{D}(A)$ and $L^2\big(\R,(1+\lambda^2)\mu_g^{(A)}\big)$ are isometrically isomorphic.
\end{proof}
%
%Moreover, we have that $\mathfrak{L} \cap \mathcal{D}(A) \cong L^2(\R,(1+\lambda^2)\mu_g^{(A)})$ in the space $V$, i.e., there exists an isometric isomorphism between $L^2(\R,(1+\lambda^2)\mu_g^{(A)})$ and $\mathfrak{L}\cap\mathcal{D}(A)$ in the Hilbert space $\mathcal{D}(\cH)$ equipped with the graph scalar product $\scalarV{\cdot}{\cdot}$ and graph norm $\norm{\cdot}{A}$. Indeed $h \in  L^2(\R,(1+\lambda^2)\mu_g^{(A)})$ if and only if $h(A)g \in \mathcal{D}(A)$ which is equivalent to the requirement $h(A)g \in \mathfrak{L} \cap \mathcal{D}(A)$. 

\begin{remark}\label{rem:HMP_iffKcoreKapprox}
According to Theorem~\ref{th:Hamburger_determinancy}, $\mathfrak{L}(A,g) = \Kc{A}{g}$ if the Hamburger moment problem generated by \eqref{eq:Lmu} and \eqref{eq:Lmu_s} is determinate for self-adjoint $A$ and $g \in C^\infty(A)$. Moreover, from Theorem~\ref{th:Hamburger_determinancy} and Proposition~\ref{prop:Kcore_Hamburgerderteminate}, if the Hamburger moment problem is determinate, then $\mathfrak{L}(A,g) \cap \mathcal{D}(A) = \Kc{A}{g} \cap \mathcal{D}(A) = \Kc{A}{g}^V$, i.e., the Krylov core condition is satisfied. Together with Proposition~\ref{prop:L2isomorphism_frakL}, we see that when the Hamburger moment problem is determinate, there exists an isometric isomorphism $\Kc{A}{g}^V \cong L^2\big(\R,(1+\lambda^2)\mu_g^{(A)}\big)$ in $V$.

Conversely we note that, for self-adjoint $A$ and $g \in C^\infty(A)$, if $\Kc{A}{g} = \mathfrak{L}(A,g)$ \emph{and} the Krylov core condition $\Kc{A}{g}^V = \Kc{A}{g} \cap \mathcal{D}(A)$ is satisfied, then the Hamburger moment problem is determinate. Indeed, from Proposition~\ref{prop:L2isomorphism_frakL} and the condition $\Kc{A}{g} = \mathfrak{L}(A,g)$,  it is immediate that 
\[\Kc{A}{g}^V \cong L^2\big(\R, (1+\lambda^2)\mu_g^{(A)} \big)\]
in $V$. Therefore given any $f \in L^2\big(\R, (1+\lambda^2)\mu_g^{(A)} \big)$ we have that $f(A)g$ is approximable by a sequence of polynomials $(p_n)_{n \in \N}$ in $V$. By the functional calculus
\[\int_\R \vert p_n(\lambda) - f(\lambda) \vert^2 (1+\lambda^2) \, \rd \mu_g^{(A)} = \Vert p_n(A)g - f(A)g \Vert_A^2 \xrightarrow{n \to \infty} 0\,,\]
so that $p_n \xrightarrow{n \to \infty} f$ in $L^2\big(\R, (1+\lambda^2)\mu_g^{(A)} \big)$. Therefore the polynomials on $\R$ are dense in $L^2\big(\R, (1+\lambda^2)\mu_g^{(A)} \big)$, and so by Corollary~\ref{cor:L2density_Hamburgerdeterminacy} the Hamburger moment problem is determinate.

Therefore, we have revealed that the condition $\Kc{A}{g} = \mathfrak{L}(A,g)$ coupled with the Krylov core condition are tantamount to the determinacy of the Hamburger moment problem for the Hamburger moment functional generated by the measure $\mu_g^{(A)}$ in \eqref{eq:Lmu}. This further underscores the intimate connection between the structural properties of the Krylov subspace $\K{A}{g}$ and the Hamburger moment problem in the self-adjoint setting.
\end{remark}

We now demonstrate appropriate limiting sequences for approximating the subspaces $\mathfrak{L}(A,g)$ and $\Kc{A}{g}$ in terms of the weak-gap metric $\hat{d}_w$ from \cite{CM-krylov-perturbation-2021} when we have that $\cH$ is separable.

\begin{proposition}\label{prop:KntoK}
Let $A$ be self-adjoint on a separable Hilbert space $\cH$, and let $g \in C^\infty(A)$. Then there exists a sequence of bounded vectors $(g_n)_{n \in \N}$ such that $g_n \xrightarrow[n \to \infty]{\Vert\cdot\Vert_{A}} g$, and moreover
	\begin{itemize}
		\item[(i)] $\Kc{A}{g_n} \subset \Kc{A}{g_{n+1}} \subset \mathfrak{L}(A,g)$ for all $n \in \N$,
		\item[(ii)] and $\Kc{A}{g_n} \xrightarrow{\hat{d}_w} \mathfrak{L}(A,g)$, where $\hat{d}_w$ is the weak-gap metric based on the weak topology of $B_\cH$.
	\end{itemize}
\end{proposition}

\begin{proof}
We begin by proving part (i). Let 
\[g_n := \chi_{[-n,n]}(A)g = \int_{\R} \chi_{[-n,n]}(\lambda) \, \rd E^{(A)}(\lambda)g\,,\] where $E^{(A)}$ is the unique projection valued spectral measure for the self-adjoint operator $A$. By Lemma~\ref{lem:boundedcore} we know that $g_n \xrightarrow{\Vert\cdot\Vert_{A}} g$.

We show that $\Kc{A}{g_n} \subset \Kc{A}{g_{n + 1}}$, and $\Kc{A}{g_n} \subset \mathfrak{L}(A,g)$ for all $n \in \N$. As $g_n \in \mathcal{D}^b(A)$, $\Kc{A}{g_n} \cong L^2\big(\R,\mu_{g_n}^{(A)}\big)$ for all $n \in \N_0$, and given any $h \in L^2\big(\R,\mu_{g_{n}}^{(A)}\big)$ 
\begin{align*}
\int_{\R} \vert h(\lambda) \vert^2 \, \rd \mu_{g_{n}}^{(A)} & = \int_{\R} \vert h(\lambda) \vert^2 \chi_{[-n,n]}(\lambda) \, \rd \mu_g^{(A)} \\
 & = \int_{\R} \vert h(\lambda) \vert^2 \chi_{[-n,n]}(\lambda) \chi_{[-n - 1,n + 1]}(\lambda)\, \rd \mu_g^{(A)} \\
 & = \int_{\R} \vert h(\lambda) \vert^2 \chi_{[-n,n]}(\lambda)\, \rd \mu_{g_{n + 1}}^{(A)}
\end{align*}
so that $h(\lambda)\chi_{[-n,n]}(\lambda) \in L^2\big(\R,\mu_{g_{n + 1}}^{(A)}\big)$. We see that 
\begin{align*}
\Kc{A}{g_{n+1}} \ni h(A) & \chi_{[-n,n]}(A)g_{n + 1}  \\
 & = h(A)\chi_{[-n,n]}(A)\chi_{[-n-1,n+1]}(A)g = h(A)g_n\,.
\end{align*}
We also have that $h(A)g_n = h(A)\chi_{[-n,n]}(A)g$ implies that $h(\lambda)\chi_{[-n,n]}(\lambda) \in \\ L^2\big(\R,\mu_g^{(A)}\big)$ and thus $h(A)g_n \in \mathfrak{L}(A,g)$ for all $n \in \N$.

We now prove part (ii). By Lemma~\ref{lem:nestedincreasingsets_convergence} we know that there exists some $M \subset B_\cH$ such that $\mathcal{K}_n \xrightarrow{\hat{d}_w} M$, where we use the shorthand $\mathcal{K}_n := \Kc{A}{g_n} \cap B_\cH$ and $\mathcal{L} := \mathfrak{L}(A,g) \cap B_\cH$.

We claim that $M = \mathcal{L}$, and is therefore the unit ball of a closed linear subspace. Indeed, from Theorem~\ref{th:fundamentalproperties_dwhat}(iii), we have
\[M = \{u \in B_\cH \,|\, u_n \rightharpoonup u \text{ for a sequence } (u_n)_{n \in \N} \text{ with } u_n \in \mathcal{K}_n\}\,.\]
Let $u \in \mathcal{L}$ be given, so there exists some $f \in L^2\big(\R,\mu_g^{(A)}\big)$ such that $f(A)g = u$. Let $f_n(\lambda) := f(\lambda)\chi_{[-n,n]}(\lambda)$, so that $f_n(A)g = f_n(A)g_n \in \mathfrak{L}(A,g_n)$ which implies $f_n(A)g \in \mathcal{K}_n$ as $\Vert f_n(A)g\Vert_\cH \leqslant \Vert f(A)g\Vert_\cH \leqslant 1$, and $\mathfrak{L}(A,g_n) = \Kc{A}{g_n}$ for all $n \in \N$ owing to the boundedness of the $g_n$'s (Proposition~\ref{prop:quasianalytic_determinate}). Letting $u_n := f_n(A)g$,
\[\Vert u_n - u\Vert_\cH^2 = \int_{\R} \vert f_n(\lambda) - f(\lambda) \vert^2 \, \rd \mu_{g}^{(A)} \xrightarrow{n \to \infty} 0\,,\]
from the Lebesgue Dominated Convergence Theorem. Therefore, there exists a sequence $(u_n)_{n \in \N}$ such that $u_n \in \mathcal{K}_n$ for all $n \in \N$ and $u_n \to u$ strongly (thus weakly), from which $u \in M$. This gives us the inclusion $\mathcal{L} \subset M$.

It is clear that $M \subset \mathcal{L}$ given that $\mathcal{K}_n \subset \mathcal{L}$ as proven above. Thus $M = \mathcal{L}$, and this completes the proof.
\end{proof}

\begin{remark}\label{rem:KrylovInnerApprox_L2approx}
This proposition shows us that when $\mathfrak{L}(A,g) = \Kc{A}{g}$ the Krylov subspace can be ``inner approximated'' by a sequence of nested Krylov subspaces. Indeed, it was shown in \cite{CM-krylov-perturbation-2021} that there always exists the property of inner approximability of Krylov subspaces when $A$ is a \emph{bounded} linear operator. 

More interestingly, it also shows us that by using a nested sequence of Krylov subspaces we can actually approximate the full space $\mathfrak{L}(A,g)$ which could strictly contain the closed Krylov subspace, precisely the case when the Hamburger moment problem \eqref{eq:Lmu} and \eqref{eq:Lmu_s} is not determinate. This suggests that such an approximation scheme of Krylov subspaces could be advantageous in constructing or approximating vectors that may not be present in $\Kc{A}{g}$ when the Hamburger moment problem is indeterminate.
\end{remark}

\begin{corollary}\label{cor:Korthog_limits}
Consider $A$, $g$, and $\cH$ as in the statement of Proposition~\ref{prop:KntoK} with the corresponding sequence $(g_n)_{n \in \N}$ from the same proposition. Then $\K{A}{g_n}^\perp \xrightarrow{\hat{d}_w} \mathfrak{L}(A,g)^\perp$.
\end{corollary}

\begin{proof}
This is a result of combining Propositions~\ref{prop:KntoK} and \ref{prop:weakconvergence_orthog}.
\end{proof}

\begin{remark}\label{rem:PKnprojector}
At this point we comment on the action of the orthogonal projection operator $P_{\mathcal{K}_n}:\cH \to \cH$ onto $\Kc{A}{g_n}$ from Proposition~\ref{prop:KntoK} and the proof therein for the choice of vectors $g_n := \chi_{[-n,n]}(A)g$ for all $n \in \N$.

For any vector $v \in \mathfrak{L}(A,g)$ it turns out that we have an explicit formula for $P_{\mathcal{K}_n}v$ that can be written in terms of the functional calculus. As there is a corresponding $h \in L^2\big(\R,\mu_g^{(A)}\big)$ such that $h(A)g = v$ we have
\begin{equation}\label{eq:rem_Pknprojector}
P_{\mathcal{K}_n}v = \int_\R h(\lambda) \chi_{[-n,n]}(\lambda) \, \rd E^{(A)}g = h(A)g_n\,.
\end{equation}
It is clear that $P_{\mathcal{K}_n}v \xrightarrow{n \to \infty} v \in \mathfrak{L}(A,g)$ by a Lebesgue Dominated Convergence argument, in agreement with Proposition~\ref{prop:convergenceofprojections}. The projection $P_{\mathcal{K}_n}v$ vector is the unique argument that minimises the square distance $\Vert u - v\Vert_\cH^2$ over all $u \in \Kc{A}{g_n}$. Indeed,
\begin{align*}
\Vert \tilde{h}(A)g_n - h(A)g\Vert_\cH^2 & = \Vert \tilde{h}(A)\chi_{[-n,n]}(A) g - h(A)g\Vert_\cH^2 \\
 & = \int_\R \vert \tilde{h}(\lambda)\chi_{[-n,n]}(\lambda) - h(\lambda)\vert^2 \, \rd \mu_g^{(A)}\\
 & = \int_{[-n,n]} \vert \tilde{h}(\lambda) - h(\lambda) \vert^2 \, \rd \mu_g^{(A)} + \int_{\R \setminus [-n,n]} \vert h(\lambda) \vert^2 \, \rd \mu_g^{(A)}\,,
\end{align*}
where $\tilde{h} \in L^2\big(\R,\mu_{g_n}^{(A)}\big)$ is the representation of any vector $u \in \Kc{A}{g_n} \cong L^2\big(\R,\mu_{g_n}^{(A)}\big)$. The minimiser for the above integral is $\tilde{h}(\lambda) = \chi_{[-n,n]}(\lambda)h(\lambda)$ $\mu_{g}^{(A)}$-a.e., i.e., $\tilde{h}(\lambda) = h(\lambda)$ for $\mu_{g}^{(A)}$-a.e. $\lambda \in [-n,n]$. This implies $P_{\mathcal{K}_n} v = h(A)g_n = \chi_{[-n,n]}(A)v$.
\end{remark}

As $V$ is a Hilbert space and separable when $\cH$ is separable, we also have the notion of the weak-gap based on the weak topology of $B_V$ in $V$, namely $\hat{d}^V_w$. We shall exploit this construction for our next theorem.

\begin{theorem}\label{th:Kreduced}
Let $A$ be a self-adjoint operator on a separable Hilbert space $\cH$, and let $g \in C^\infty(A)$ be such that the Hamburger moment problem \eqref{eq:Lmu} and \eqref{eq:Lmu_s} is determinate (for example, $g \in \mathcal{D}^{qa}(A)$).
	\begin{itemize}
		\item[(i)] $\Kint{A}{g} = \{0\}$. Moreover, if $g \in \ran A$ then there exists a solution $f$ to $Af =g$ in $\Kc{A}{g}^V$.
		\item[(ii)] For the sequence of vectors $(g_n)_{n \in \N}$ as in the statement of Proposition~\ref{prop:KntoK}, we have that $\Kc{A}{g_n}^V \xrightarrow{\hat{d}^V_w} \Kc{A}{g}^V$ and $\K{A}{g_n}^{\perpV} \xrightarrow{\hat{d}^V_w} \K{A}{g}^{\perpV}$, where $\hat{d}^V_w$ is the weak-gap metric based on the unit ball $B_V = \{v \in V \,|\, \Vert v\Vert_{A} \leqslant 1\}$ in the Hilbert space $V$.
	\end{itemize} 
\end{theorem}

\begin{proof}
We begin by proving statement (ii). Consider the sequence of bounded vectors $(g_n)_{n \in \N}$ corresponding to the same sequence of vectors in the statement of Proposition~\ref{prop:KntoK}. It is known that $\Kc{A}{g} \cong L^2\big(\R,\mu_g^{(A)}\big)$, that is, $\mathfrak{L}(A,g) = \Kc{A}{g}$, as the Hamburger moment problem is determinate (Theorem~\ref{th:Hamburger_determinancy}).

We already know that $\Kc{A}{g_n}$ and $\Kc{A}{g}$ satisfy the Krylov core condition (Proposition~\ref{prop:Kcore_Hamburgerderteminate}) and $\Kc{A}{g_n} \subset \Kc{A}{g}$ (Proposition~\ref{prop:KntoK}), so that we end up with the inclusions
\[\Kc{A}{g_n}^V \subset \Kc{A}{g_{n+1}}^V \subset \Kc{A}{g}^V\,, \quad \forall \, n \in \N\,,\]
from which
\[\K{A}{g}^{\perpV} \subset \K{A}{g_{n+1}}^{\perpV} \subset \K{A}{g_{n}}^{\perpV}\,, \quad \forall \, n \in \N\,.\]
The nested structure of the Krylov subspaces means that $\Kc{A}{g_n}^V \cap B_V \xrightarrow{\hat{d}^V_w} C$ where $C$ is a weakly closed subset of $B_V$ (Lemma~\ref{lem:nestedincreasingsets_convergence}). Owing to the set inclusions above, we see that $C \subset \Kc{A}{g}^V \cap B_V$. Let $p$ be any polynomial, and by Lemma~\ref{lem:p(A)gn_to_p(A)g} we have
\[\frac{\alpha p(A)g_n}{\Vert p(A)g\Vert_{A}} \xrightarrow[n \to \infty]{\Vert \cdot\Vert_{A}} \frac{\alpha p(A)g}{\Vert p(A)g\Vert_{A}}\,,\]
where $\frac{\alpha p(A)g_n}{\Vert p(A)g\Vert_{A}} \in B_V$ for all $n \in \N$, $\vert \alpha \vert \leqslant 1$. This shows us that $\K{A}{g} \cap B_V \subset C$, and thus $\Kc{A}{g}^V \cap B_V \subset C$, which implies that $C = \Kc{A}{g}^V \cap B_V$ and therefore $\Kc{A}{g_n}^V \xrightarrow{\hat{d}^V_w} \Kc{A}{g}^V$. Moreover combining this with Proposition~\ref{prop:weakconvergence_orthog} gives us the convergence $\K{A}{g_n}^{\perpV} \xrightarrow{\hat{d}^V_w} \K{A}{g}^{\perpV}$.

We now prove statement (i). Due to the inclusions, we see that 
\[A \big(\K{A}{g}^{\perpV}\big) \subset A \big(\K{A}{g_n}^{\perpV}\big) \subset \K{A}{g_n}^\perp\,,\]
where the last inclusion is from Proposition~\ref{prop:bddvec_invariances}. Owing to the convergence of the orthogonal complements (in $\cH$) established in Corollary~\ref{cor:Korthog_limits} and $\mathfrak{L}(A,g) = \Kc{A}{g}$, taking the limit as $n \to \infty$ we see that
\[A\big(\K{A}{g}^{\perpV}\big) \subset \K{A}{g}^\perp\,,\]
and therefore we have that $\Kint{A}{g} = \{0\}$.
\end{proof}

\begin{remark}\label{rem:KnVtoKV}
We note that should we release the assumption of Theorem~\ref{th:Kreduced} that the Hamburger moment problem be determinate, we can still say something informative about the weak-gap convergence of the nested sequence of Krylov subspaces $(\Kc{A}{g_n}^V)_{n \in \N}$ in $V$.

In such a case we have that $\Kc{A}{g_n}^V \subset \Kc{A}{g_{n+1}}^V \subset \mathfrak{L}(A,g) \cap \mathcal{D}(A) \cong L^2\big(\R,(1+\lambda^2)\mu_g^{(A)}\big)$ for all $n \in \N$. By Lemma~\ref{lem:nestedincreasingsets_convergence} there exists $C \subset B_V$ weakly closed (in $V$) such that $\Kc{A}{g_n} \xrightarrow{\hat{d}^V_w} C$. It is also clear that $C \subset \mathfrak{L}(A,g) \cap B_V = \{v \in \mathfrak{L}(A,g)\cap\mathcal{D}(A) \, | \, \Vert v\Vert_{A} \leqslant 1\}$. Taking any $h \in L^2\big(\R,(1+\lambda^2)\mu_g^{(A)}\big)$ we let $h_n(\lambda) := h(\lambda)\chi_{[-n,n]}(\lambda)$ so that $\vert h_n \vert \leqslant \vert h \vert$ and $h_n(A)g = h(A)g_n$ for all $n \in \N$. Then $\alpha\frac{\Vert h(A)g_n\Vert_{A}}{\Vert h(A)g\Vert_{A}} \leqslant 1$ for all $\vert \alpha \vert \leqslant 1$ and due to Lebesgue Dominated Convergence we have
\[\Vert h(A)g_n - h(A)g\Vert_{A}^2 \xrightarrow{n\to\infty} 0\,,\]
so that clearly $\alpha \frac{h(A)g}{\Vert h(A)g\Vert_{A}} \in C$ for all $\vert \alpha \vert \leqslant 1$. This gives the inclusion $\mathfrak{L}(A,g) \cap B_V \subset C$. Therefore,
\begin{equation}\label{eq:remark_KnVtoKV}
\Kc{A}{g_n}^V \xrightarrow{\hat{d}^V_w} \mathfrak{L}(A,g) \cap \mathcal{D}(A)\,.
\end{equation}
\end{remark}

%\section{Some Examples}
%In this section, we provide some examples that illustrate the criticality of some of our assumptions made in the previous section.

%\section{Acknowledgements}\label{sec:acknowledgements}

\bibliographystyle{siam}
\def\cprime{$'$}

\end{document}